\documentclass{article}
\usepackage{amsmath}
\usepackage{amsthm}
\usepackage{amsfonts}
\usepackage{amssymb}
\usepackage[all]{xy}
\usepackage{url}
\usepackage{xr}

\DeclareMathSymbol{\rightrightarrows}  {\mathrel}{AMSa}{"13}

\def\Ob{\operatorname{Ob}}

\def\ho{\operatorname{ho}}

\def\sk{\operatorname{sk}}

\def\Iso{\operatorname{Iso}}

\def\Path{\operatorname{Path}}

\catcode`\@=11
\def\varholim@#1#2{\mathop{\vtop{\ialign{##\crcr
 \hfil$#1\m@th\operator@font holim$\hfil\crcr
 \noalign{\nointerlineskip\kern\ex@}#2#1\crcr
 \noalign{\nointerlineskip\kern-\ex@}\crcr}}}}
\def\hocolim{\mathpalette\varholim@\rightarrowfill@} 
\def\hoinvlim{\mathpalette\varholim@\leftarrowfill@}
\catcode`\@=\active

\newtheorem{theorem}{Theorem}%[section]
\newtheorem{lemma}[theorem]{Lemma}
\newtheorem{corollary}[theorem]{Corollary}

\newtheorem{proposition}[theorem]{Proposition}

\theoremstyle{definition}

\newtheorem{example}[theorem]{Example}

\newtheorem{remark}[theorem]{Remark}

\begin{document}

\title{Path categories and quasi-categories}

\author{J.F. Jardine\thanks{This research was supported by NSERC.}}

%\date{December, 2015}

\maketitle

\begin{abstract}
This paper gives an introduction to the
homotopy theory of quasi-categories. 
Weak equivalences between quasi-categories are characterized as maps which
induce equivalences on a naturally defined system of groupoids. These
groupoids effectively replace higher homotopy groups in quasi-category
homotopy theory. 
\end{abstract}

\section*{Introduction}

The identification of morphism sets in path categories of simplicial
(or cubical) complexes is a central theme of concurrency theory. The
path category functor, on the other hand, plays a role in the homotopy
theory of quasi-categories that is roughly analogous to that of the
fundamental groupoid in standard homotopy theory.

Concurrency theory is the study of models of parallel processing
systems. One of the prevailing geometric forms of the theory
represents systems as finite cubical complexes, called higher
dimensional automata. Each $r$-cell of the complex $K$ represents the
simultaneous action of a group of $r$ processors, while the holes in
the complex represent constraints on the system, such as areas of
memory that cannot be shared.  The vertices of $K$ are the states of
the system, and the morphisms $P(K)(x,y)$ from $x$ to $y$ in the path
category $P(K)$ are the execution paths from the state $x$ to the
state $y$.

There is no prevailing view of what higher homotopy invariants should
mean in concurrency theory, or even what these invariants should
be. The path category functor is not an artifact of standard homotopy
theory, but it is a central feature of the theory of
quasi-categories. The homotopy theory of quasi-categories is
constructed abstractly within the category of simplicial sets by
methods that originated in homotopy localization theory, and its weak
equivalences are not described by homotopy groups.

Is the homotopy theory of quasi-categories the right venue for the
study of higher homotopy phenomena in concurrency theory? Well maybe, but
a definitive answer is not presented here. 

It is a fundamental aspect of this flavour of homotopy theory that if
a map $X \to Y$ of simplicial sets is a quasi-category weak
equivalence (a ``categorical weak equivalence'' below), then the
induced functor $P(X) \to P(Y)$ is an equivalence of categories ---
see Lemma \ref{lem 28}. This phenomenon may be a bit strong for
computational purposes, given that some complexity reduction
techniques for concurrency algorithms \cite{complexity} focus on producing subcomplexes
$L \subset K$ of a simplicial complex $K$ which induce fully faithful
imbeddings $P(L) \to P(K)$ of the associated path categories.

It nevertheless appears that the homotopy theory of quasi-categories
is a good first approximation of a theory that is suitable for the
homotopy theoretic analysis of the computing models that are
represented by higher dimensional automata. 

Further, recent work of Nicholas Meadows \cite{Meadows-01} shows that
the theory of quasi-categories can be extended to a homotopy theory for
simplicial presheaves whose weak equivalences are those maps which are
stalkwise quasi-category weak equivalences. Meadows' theory starts to give
a local to global picture of systems that are infinite in the
practical sense that they cannot be studied by ordinary algorithms.
\medskip              

This paper is esssentially expository: it is an introduction to the
``ordinary'' homotopy theory of quasi-categories. The exposition
presented here has evolved from research notes that were written at a
time when complete descriptions of Joyal's quasi-category model
structure for simplicial sets were not publicly available. The overall
line of argument is based on methods which first appeared Cisinski's
thesis \cite{Cisinski}, \cite{J40}, and is now standard. This
presentation is rapid, and is at times aggressively
combinatorial. See also \cite{Joyal-quasi-Kan} and \cite{Lurie-HTT}.

Proposition \ref{prop 38} characterizes
weak equivalences between quasi-categories as maps which
induce equivalences on a naturally defined system of groupoids. These
groupoids effectively replace higher homotopy groups in quasi-category
homotopy theory. 

The collection of ideas appearing in the argument for Proposition
\ref{prop 38} also applies within the ordinary homotopy theory of
simplicial sets. Corollary \ref{cor 39} says that there is a system of
naturally defined groupoids for Kan complexes, such that a map of Kan
complexes is a standard weak equivalence if and only if it induces
equivalences of this list of associated groupoids. This criterion,
along with the argument for it, is base point free. Corollary 
\ref{cor 39} can also be proved directly with
standard techniques of simplicial homotopy theory.
\medskip

\noindent
{\bf NB}:\ This is a lightly revised version of an internal document that was written in 2015, in support of 
the Meadows paper \cite{Meadows-01} and the thesis \cite{Meadows-thesis} that followed. This paper will not be published in its present form.

\vfill\eject

\tableofcontents

\section{The path category functor}

Suppose that $X$ is a simplicial set. The {\it path category} $P(X)$
of $X$ is the category freely that is generated by the graph defined
by the vertices and edges of $X$, subject to the relations
\begin{equation*}
d_{1}(\sigma) = d_{0}(\sigma)\cdot d_{2}(\sigma),
\end{equation*}
one for each $2$-simplex $\sigma$ of $X$, and $s_{0}(x) = 1_{x}$ for
each vertex $x$ of $X$.

The category $P(X)$ has morphism sets $P(X)(x,y)$ defined by
equivalence classes of strings of $1$-simplices
\begin{equation*}
\alpha:\ x = x_{0} \xrightarrow{\alpha_{1}} x_{1} \xrightarrow{\alpha_{2}} \dots
\xrightarrow{\alpha_{n}} x_{n} = y,
\end{equation*}
where the equivalence relation is generated by relations of the form $\alpha
\sim \alpha'$ occurring in the presence of a $2$-simplex $\sigma$ of $X$ with
boundary
\begin{equation*}
\xymatrix@C=8pt{
x_{i} \ar[rr]^{\beta} \ar[dr]_{\alpha_{i+1}} && x_{i+2} \\
& x_{i+1} \ar[ur]_{\alpha_{i+2}}
}
\end{equation*}
and where $\alpha'$ is the string
\begin{equation*}
\alpha':\ x_{0} \xrightarrow{\alpha_{1}} \dots \xrightarrow{\alpha_{i}} 
x_{i} \xrightarrow{\beta} x_{i+2} \xrightarrow{\alpha_{i+3}} \dots \xrightarrow{\alpha_{n}} x_{n}.
\end{equation*}
Composition in $P(X)$ is defined by concatenation of representing strings.

The resulting functor $X
\mapsto P(X)$ is left adjoint to the nerve functor
\begin{equation*}
B: \mathbf{cat} \to s\mathbf{Set},
\end{equation*}
essentially since the nerve functor takes values in
$2$-coskeleta. The inclusion $\sk_{2}(X) \subset X$ induces
an isomorphism of categories
\begin{equation*}
P(\sk_{2}(X)) \cong P(X).
\end{equation*}

Given a simplicial set map $f: X \to BC$, the adjoint functor
$f_{\ast}: P(X) \to C$ is the map $f: X_{0} \to \Ob(C)$ in degree $0$,
and that takes a $1$-simplex $\alpha: d_{1}(\alpha) \to d_{0}(\alpha)$ to
the morphism $f(\alpha): f(d_{1}(\alpha)) \to f(d_{0}(\alpha))$ of
$C$. In particular, the canonical functor $\epsilon: P(BC) \to C$ is
the identity on objects, and takes a $1$-simplex $\alpha: x \to y$ of
$BC$ to the corresponding morphism of $C$.

\begin{lemma}\label{lem 1}
The canonical functor 
\begin{equation*}
\epsilon: P(BC) \to C
\end{equation*}
is an isomorphism for each small category $C$.
\end{lemma}

\begin{proof}
There is a functor $s: C \to P(BC)$, which is the identity on objects,
and takes a morphism $\alpha: x \to y$ to the morphism that is represented by
the $1$-simplex $\alpha$ of $BC$. The composition law is preserved,
because the $2$-simplices of $BC$ are composition laws. The morphism
of $P(BC)$ that is represented by a string of $1$-simplices
\begin{equation*}
x =x_{0} \xrightarrow{\alpha_{1}} x_{1} \xrightarrow{\alpha_{2}} \dots
\xrightarrow{\alpha_{n}} x_{n}=y 
\end{equation*}
in $BC$ is also represented by the
composite $\alpha_{n} \cdots \alpha_{1}$, and it follows
that the functor $s$ is full. But $\epsilon \cdot s = 1_{C}$, so that
$s$ is faithful as well as full, and is therefore an isomorphism of
categories.
\end{proof}

\begin{lemma}\label{lem 2}
The canonical functor
\begin{equation*}
P(X \times Y) \to P(X) \times P(Y)
\end{equation*}
is an isomorphism for all simplicial sets $X$ and $Y$.
\end{lemma}

\begin{proof}
If $C$ and $D$ are small categories, then there is a diagram
\begin{equation*}
\xymatrix{
P(B(C \times D)) \ar[r] \ar[d]_{\epsilon}^{\cong} & P(BC) \times P(BD) 
\ar[d]^{\epsilon \times \epsilon}_{\cong} \\
C \times D \ar[r]_{1} & C \times D
}
\end{equation*}
by Lemma \ref{lem 1}, so the claim holds for $X=BC$ and $Y=BD$. In
general $X \times Y$ is a colimit of products
\begin{equation*}
\Delta^{n} \times \Delta^{m} = B(\mathbf{n}) \times B(\mathbf{m}), 
\end{equation*}
and so the comparison map is a colimit of the isomorphisms
\begin{equation*}
P(\Delta^{n} \times \Delta^{m}) \xrightarrow{\cong} P(\Delta^{n}) \times P(\Delta^{m}).
\end{equation*}
\end{proof}

\begin{corollary}\label{cor 3}
\begin{itemize}
\item[1)] 
Every simplicial homotopy $h: X \times \Delta^{1} \to Y$
  between maps $f,g: X \to Y$ induces a natural transformation
  $h_{\ast}: P(X) \times \mathbf{1} \to P(Y)$ between the
  corresponding functors $f_{\ast},g_{\ast}: P(X) \to P(Y)$.
\item[2)]
Every simplicial homotopy equivalence $X \to Y$ induces a homotopy equivalence
of categories $P(X) \to P(Y)$.
\item[3)]
Every trivial Kan fibration $\pi: X \to Y$ induces a strong
deformation retraction $\pi_{\ast}: P(X) \to P(Y)$ of $P(X)$ onto
$P(Y)$. The functor $\pi_{\ast}$ is a homotopy equivalence of categories.
\end{itemize}
\end{corollary}

A functor $f: C \to D$ between small categories is a {\it homotopy
equivalence} if the induced map $BC \to BD$ is a homotopy
equivalence. This means that there is a functor $g: D \to C$ and
natural transformations $C \times \mathbf{1} \to C$ and $D \times
\mathbf{1} \to D$ between $g\cdot f$ and $1_{C}$ and between $f \cdot
g$ and $1_{D}$. Note that the direction of these natural
transformations is not specified.

\begin{proof}[Proof of Corollary \ref{cor 3}]
Statement 1) is a consequence of Lemma \ref{lem 1} and Lemma \ref{lem 2}.

For statement 2), we are assuming the existence of a simplicial set
map $g: Y \to X$, together with simplicial homotopies $X \times
\Delta^{1} \to X$ between $gf$ and $1_{X}$ and $Y \times \Delta^{1}
\to Y$ between $fg$ and $1_{Y}$. The claim follows from
statement 1).

To prove statement 3), it is a standard observation that $\pi$ has a
section $\sigma: Y \to X$ along with a simplicial homotopy $h: X
\times \Delta^{1} \to X$ between $\sigma \pi$ and $1_{X}$, so that $Y$
is a strong deformation retract of $X$. More explicitly, $h$ is
constructed by finding the lifting in the diagram
\begin{equation*}
\xymatrix{
X \times \partial\Delta^{1} \ar[rr]^{(1_{X},\sigma\pi)} \ar[d] && X \ar[d]^{\pi} \\
X \times \Delta^{1} \ar[r]_{\pi \times 1} \ar@{.>}[urr]^{h} 
& Y \times \Delta^{1} \ar[r]_{pr} & Y
}
\end{equation*}
where the projection $pr$ is a constant homotopy. If $x \in X_{0}$, then the composite
\begin{equation*}
\Delta^{1} \xrightarrow{(x,1)} X \times \Delta^{1} \xrightarrow{h} X
\end{equation*}
is a $1$-simplex of the fibre $F_{\pi(x)}$ over $\pi(x)$. This fibre
is a Kan complex, so that the path is invertible in $P(F_{\pi(x)})$
and hence in $P(X)$. It follows that the induced natural transformation
\begin{equation*}
h_{\ast}: P(X) \times \mathbf{1} \to P(X)
\end{equation*}
is a natural isomorphism.
\end{proof}

The set $\pi(X,Y)$ of naive homotopy classes between simplicial sets
$X$ and $Y$ is the set of path components of the function space
$\mathbf{hom}(X,Y)$. Generally the set of path components $\pi_{0}(X)$
of a simplicial set $X$ coincides with the set
\begin{equation*}
\pi_{0}P(X) = \pi_{0}(B(P(X)))
\end{equation*}
of path components of the path category of $X$. We shall be interested
in something stronger, namely the set
\begin{equation*}
\tau_{0}P(X)
\end{equation*}
of isomorphism classes of $P(X)$. 

Say that maps $f,g: X \to Y$ are
{\it strongly homotopic} if there is an isomorphism $f
\xrightarrow{\cong} g$ in the path category $P(\mathbf{hom}(X,Y))$, and write
\begin{equation*}
\tau_{0}(X,Y) = \tau_{0}P(\mathbf{hom}(X,Y)).
\end{equation*}

Suppose that $X$ is a Kan complex. Then a morphism of $P(X)$ that
is represented by a string of $1$-simplices
\begin{equation*}
x_{0} \xrightarrow{\alpha_{1}} x_{1} \xrightarrow{\alpha_{2}} \dots
\xrightarrow{\alpha_{n}} x_{n},
\end{equation*}
is also represented by a $1$-simplex $x_{0} \to x_{n}$. In effect, the
string
\begin{equation*}
x_{0} \xrightarrow{\alpha_{1}} x_{1} \xrightarrow{\alpha_{2}} x_{2}
\end{equation*}
defines a simplicial set map $\Lambda^{2}_{1} \to X$ that extends to
a $2$-simplex $\sigma: \Delta^{2} \to X$, and so the morphism
represented by the string $(\alpha_{1},\alpha_{2})$ is also
represented by the $1$-simplex $d_{1}(\sigma): x_{0} \to x_{2}$.

\begin{lemma}\label{lem 4}
Suppose that $X$ is a Kan complex. Then $P(X)$ is a groupoid.
\end{lemma}

\begin{proof}
Every morphism of $P(X)$ is represented by a $1$-simplex $\alpha: x
\to y$ of $X$. The solutions of the lifting problems
\begin{equation*}
\xymatrix@C=40pt{
\Lambda^{2}_{0} \ar[r]^{(\ ,s_{0}(x),\alpha)} \ar[d] & X \\
\Delta^{2} \ar@{.>}[ur]
}
\end{equation*}
and
\begin{equation*}
\xymatrix@C=40pt{
\Lambda^{2}_{2} \ar[r]^{(\alpha,s_{0}(y),\ )} \ar[d] & X \\
\Delta^{2} \ar@{.>}[ur]
}
\end{equation*}
imply that there are $1$-simplices $g: y \to x$ and $h: y \to x$
(respectively) such that $g \cdot \alpha = 1_{x}$ and $\alpha \cdot h
= 1_{y}$ in $P(X)$. But then
\begin{equation*}
g = g \cdot \alpha \cdot h = h
\end{equation*}
so that $\alpha$ has an inverse in $P(X)$.
\end{proof}

Joyal shows in \cite{Joyal-quasi-Kan} (see also Corollary \ref{cor 17}
below) that if $X$ is a quasi-category (see Section 2), then the path
category $P(X)$ is a groupoid if and only if $X$ is a Kan complex.

\begin{remark}
Lemmas \ref{lem 1} and \ref{lem 4} together imply that the functor $X
\mapsto P(X)$ does not preserve standard weak equivalences. In other words, it
is not the case that a weak equivalence $f: X \to Y$ necessarily
induces a weak equivalence $BP(X) \to BP(Y)$. Otherwise, the natural
weak equivalence $X \simeq B(\mathbf{\Delta}/X)$ and any fibrant model
$j: X \to Z$ ($Z$ a Kan complex) would together give weak equivalences
\begin{equation*}
X \simeq B(\mathbf{\Delta}/X) \simeq BP(X) \simeq BP(Z),
\end{equation*}
while $BP(Z)$ has only trivial homotopy groups in degrees above $1$.
\end{remark}

\begin{lemma}\label{lem 6}
Suppose that $i: \Lambda^{n}_{k} \subset \Delta^{n}$ is the standard
inclusion, where $n \geq 2$ and $k \ne 0,n$. Then the induced functor
\begin{equation*}
i_{\ast}: P(\Lambda^{n}_{k}) \to P(\Delta^{n})
\end{equation*}
is an isomorphism.
\end{lemma}

\begin{proof}
Recall that the canonical map 
\begin{equation*}
P(\Delta^{n}) = P(B(\mathbf{n})) \xrightarrow{\epsilon} \mathbf{n}
\end{equation*}
is an isomorphism. A functor $\alpha: \mathbf{n} \to C$ can be
identified with the string of arrows
\begin{equation*}
\alpha(0) \xrightarrow{\alpha(0 \leq 1)} \alpha(1)
\xrightarrow{\alpha(1 \leq 2)} \dots \xrightarrow{\alpha(n-1 \leq n)}
\alpha(n).
\end{equation*}
It follows that the $1$-simplices $i \leq i+1: \Delta^{1} \to
\Delta^{n}$ together define a simplicial set inclusion
\begin{equation*}
\gamma_{n}: \Path_{n} = \Delta^{1} \ast \dots \ast \Delta^{1} \subset \Delta^{n}
\end{equation*}
defined on the join $\Path_{n}$ of $n$ copies of $\Delta^{1}$ (end to
end), which induces an isomorphism
\begin{equation*}
\gamma_{n\ast}: P(\Path_{n}) \xrightarrow{\cong} P(\Delta^{n}) \cong \mathbf{n}.
\end{equation*}

All $1$-simplices $i \leq i+1$ of $\Delta^{n}$ are members of
$\Lambda^{n}_{k}$ since $k \ne 0,1$: in effect, $i \leq i+1$ is a
member of $d^{n}(\Delta^{n-1})$ if $i \leq n-1$, and $ i \leq i+1$ is
in $d^{0}(\Delta^{n-1})$ if $i \geq 1$. It follows that there is a commutative diagram
\begin{equation*}
\xymatrix{
& \Lambda^{n}_{k} \ar[d]^{i} \\
\Path_{n} \ar[ur] \ar[r]_-{\gamma_{n}} & \Delta^{n}
}
\end{equation*} 
and hence a commutative diagram
\begin{equation*}
\xymatrix{
& P(\Lambda^{n}_{k}) \ar[d]^{i_{\ast}} \\
P(\Path_{n}) \ar[ur]^{\hat{\gamma}} \ar[r]_-{\gamma_{n\ast}}^-{\cong} 
& P(\Delta^{n})
}
\end{equation*} 
Every ordinal number map $\theta: \mathbf{m} \to \mathbf{n}$
determines a commutative diagram
\begin{equation*}
\xymatrix{
P(\Path_{m}) \ar[d]_{\gamma_{m\ast}}^{\cong} 
\ar[r]^{\tilde{\theta}} &
P(\Path_{n}) \ar[d]^{\gamma_{n\ast}}_{\cong} \\
P(\Delta^{m}) \ar[r]_{\theta_{\ast}} & P(\Delta^{n})
}
\end{equation*}
where $\tilde{\theta}$ sends the $1$-simplex $i \leq i+1$ to the composite
morphism
\begin{equation*}
\theta(i) \leq \theta(i)+1 \leq \theta(i)+2 \leq \dots 
\leq \theta(i+1).
\end{equation*}
It follows that every $1$-simplex $i \leq j$ of $\Lambda^{n}_{k}$,
which is in some face $d^{r}(\Delta^{n-1})$, is in the image of the
functor
\begin{equation*}
\hat{\gamma}: P(\Path_{n}) \to P(\Lambda^{n}_{k}).
\end{equation*}
The functor $\hat{\gamma}$ is therefore surjective on morphisms, and 
is an isomorphism.
\end{proof}

\section{Quasi-categories}

The class of {\it inner anodyne extensions} in simplicial sets is the
saturation of the set of morphisms
\begin{equation*}
i:\Lambda^{n}_{k} \subset \Delta^{n},\ k \ne 0,n.
\end{equation*}

A {\it quasi-category} is a simplicial set $X$ such that the map $X
\to \ast$ has the right lifting property with respect to all inner
anodyne extensions. A map $p: X \to Y$ that has the right lifting
property with respect to all inner anodyne extensions is called an
{\it inner fibration}.
\medskip

\noindent
{\bf Example}:\ Suppose that $C$ is a small category. Then $BC$ is a
quasi-category, by Lemma \ref{lem 6}.
\medskip

Here's an observation:

\begin{lemma}\label{lem 7}
Suppose that $i: A \to B$ is an inner anodyne extension. Then the map
$i: A_{0} \to B_{0}$ on vertices is a bijection.
\end{lemma}

\begin{proof}
The class of monomorphisms $j: E \to F$ that are bijections on
vertices is saturated and includes all inclusions $\Lambda^{n}_{k}
\subset \Delta^{n}$, $0 < k < n$. 

To see this last claim, observe that
every vertex $i$ of $\Delta^{n}$ is in the face $d^{j}: \Delta^{n-1}
\to \Delta^{n}$ if $j \ne i$. Thus, if $n \geq 2$ then the vertex $i$ is in
at least two faces of $\Delta^{n}$. 
\end{proof}

The following result is more serious:

\begin{lemma}\label{lem 8}
Suppose that $S$ is a proper subset of the $(n-1)$-simplices $d^{i}$ in
$\Delta^{n}$ that contains $d^{0}$ and $d^{n}$,
and let $\langle S \rangle$ be the subcomplex of $\Delta^{n}$ that is
generated by the simplices in $S$. Then the inclusion $\langle S
\rangle \subset \Delta^{n}$ is inner anodyne.
\end{lemma}

\begin{proof}
The proof is by decreasing induction on the cardinality of $S$ and
increasing induction on the dimension $n$. Observe that $2 \leq \vert
S \vert \leq n$, and $\langle S \rangle = \Lambda^{n}_{k}$ for some
$k$ if $\vert S \vert = n$.

Suppose that $S'$ is obtained from $S$ by adding a simplex
$d^{k}$, where $0 < k < n$. Then the intersection
\begin{equation*}
\langle d^{k} \rangle \cap \langle S' \rangle
\end{equation*}
is the subcomplex of $\Delta^{n-1} \cong \langle d^{k} \rangle$ that
is generated by the set $S''$ of simplices $d^{k}d^{i}$ and
$d^{k}d^{j-1}$, where the simplices $d^{i},d^{j}$ are the members of
$S'$ with $i < k$ and $j > k$, respectively. In particular, the bottom
and top faces $d^{k}d^{0}$ and $d^{k}d^{n-1}$ of $\langle d^{k}
\rangle$ are in $S''$, and $\vert S'' \vert = \vert S' \vert - 1 <
n-1$. 

There is a pushout diagram
\begin{equation*}
\xymatrix{
\langle S'' \rangle \ar[r] \ar[d]_{i} & \langle S' \rangle \ar[d] \\
\Delta^{n-1} \ar[r]_{d^{k}} & \langle S \rangle
}
\end{equation*} 
The inclusion $i$ is inner anodyne by induction on dimension, so that
the inclusion $\langle S ' \rangle \subset \langle S \rangle$ is inner
anodyne. Thus, since $\langle S \rangle \subset \Delta^{n}$ is inner
anodyne, the inclusion $\langle S' \rangle \subset \Delta^{n}$ is
inner anodyne as well.
\end{proof}

The following result is proved in Section 4 (Theorem \ref{th 45})
below. The proof is somewhat delicate, and makes heavy use of Lemma 
\ref{lem 8}.

\begin{theorem}\label{thm 9}
Suppose that $0 < k < n$. Then the inclusion
\begin{equation*}
(\Lambda^{n}_{k} \times \Delta^{m}) \cup (\Delta^{n} \times \partial\Delta^{m})
\subset \Delta^{n} \times \Delta^{m}
\end{equation*}
is inner anodyne.
\end{theorem}

\begin{remark}
The inclusion $\Delta^{0} \to \Delta^{1}$ of a vertex induces a map
\begin{equation*}
(\Delta^{0} \times \Delta^{1}) \cup (\Delta^{1} \times \partial\Delta^{1}) \subset \Delta^{1} \times \Delta^{1},
\end{equation*}
which is not inner anodyne. This observation is a consequence of Lemma
\ref{lem 6}, since the induced map of path categories is not an
isomorphism.
\end{remark} 

\begin{corollary}\label{cor 11}
\begin{itemize}
\item[1)] Suppose that $A \subset B$ is an inclusion of simplicial
  sets, and that $X$ is a quasi-category. Then the map 
\begin{equation*}
i^{\ast}:
  \mathbf{hom}(B,X) \to \mathbf{hom}(A,X) 
\end{equation*}
is an inner fibration. If
  the map $i: A \to B$ is an inner anodyne extension then the map
  $i^{\ast}$ is a trivial Kan fibration.
\item[2)] Suppose that $X$ is a quasi-category and that $K$ is a
  simplicial set. Then the function complex $\mathbf{hom}(K,X)$ is a
  quasi-category.
\end{itemize}
\end{corollary}

Suppose again that $X$ is a quasi-category, and suppose that
$\alpha,\beta: \Delta^{1} \to X$ are $1$-simplices $x \to y$. A right
homotopy $\alpha \Rightarrow_{R} \beta$ is a $2$-simplex $\sigma:
\Delta^{2} \to X$ with boundary
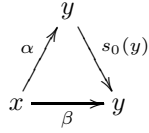
\begin{equation*}
\xymatrix@C=8pt{
& y \ar[dr]^{s_{0}(y)} & \\
x \ar[rr]_{\beta} \ar[ur]^{\alpha} && y
}
\end{equation*}
and a left homotopy $\beta \Rightarrow_{L} \alpha$ is a $2$-simplex of
$X$ with boundary
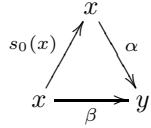
\begin{equation*}
\xymatrix@C=8pt{
& x \ar[dr]^{\alpha} & \\
x \ar[rr]_{\beta} \ar[ur]^{s_{0}(x)} && y
}
\end{equation*}
Then the following are equivalent (by suitable choices of $3$-simplices):
\begin{itemize}
\item[a)]
there is a right homotopy $\alpha \Rightarrow_{R} \beta$,
\item[b)]
there is a right homotopy $\beta \Rightarrow_{R} \alpha$,
\item[c)]
there is a left homotopy $\beta \Rightarrow_{L} \alpha$,
\item[d)]
there is a left homotopy $\alpha \Rightarrow_{L} \beta$.
\end{itemize}
If any one of these conditions holds, say that $\alpha$ is homotopic
to $\beta$, and write $\alpha \simeq \beta$.

Write $\ho(X)$ for the category whose objects are the vertices of $X$,
whose morphisms $[\alpha]: x \to y$ are the homotopy classes of paths
$\alpha: x \to y$ in $X$, and with composition law
\begin{equation*}
\ho(X)(x,y) \times \ho(X)(y,z) \to \ho(X)(x,z)
\end{equation*}
defined for classes $[\alpha]:x \to y$ and $[\beta]: y \to z$ by
\begin{equation*}
[\beta] \cdot [\alpha] = [d_{1}(\sigma)]
\end{equation*}
where $\sigma: \Delta^{2} \to X$ is a choice of extension, as in the diagram
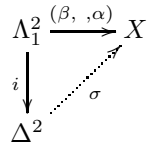
\begin{equation*}
\xymatrix{
\Lambda^{2}_{1} \ar[r]^{(\beta,\ ,\alpha)} \ar[d]_{i} & X \\
\Delta^{2} \ar@{.>}[ur]_{\sigma}
}
\end{equation*}
One must show that the class $[d_{1}(\sigma)]$ is independent of the
choices that are made.

\begin{lemma}
There is an isomorphism of categories
\begin{equation*}
P(X) \cong \ho(X)
\end{equation*}
for all quasi-categories $X$. 
\end{lemma}

\begin{proof}
The functor $P(X) \to \ho(X)$
is induced by the assignment $\alpha \mapsto [\alpha]$ for paths
$\alpha$, while all members of a homotopy class represent the same
morphism in $P(X)$. Thus, there is a functor $\ho(X) \to P(X)$, and the
two functors are inverse to each other.
\end{proof}

It follows (choose some $3$-simplices) that if $\alpha: \Delta^{1} \to
X$ is invertible in $P(X)$, where $X$ is a quasi-category, then there
is a path $\beta: \Delta^{1} \to X$ together with $2$-simplices
$\sigma,\sigma': \Delta^{2} \to X$ having respective boundaries
\begin{equation*}
{\xymatrix@C=8pt{
& y \ar[dr]^{\beta} & \\
x \ar[rr]_{s_{0}(x)} \ar[ur]^{\alpha} && x
}}\quad
{\xymatrix@C=8pt{
& x \ar[dr]^{\alpha} & \\
y \ar[rr]_{s_{0}(y)} \ar[ur]^{\beta} && y
}}
\end{equation*}
Under these circumstances, say that the $1$-simplex $\alpha$ is a
{\it quasi-isomorphism} of $X$, as is $\beta$.

In the presence of such $2$-simplices for a path $\alpha:
\Delta^{1} \to X$ in an arbitrary simplicial set $X$, the
corresponding morphism $\alpha$ is invertible in the path category
$P(X)$. 
\medskip

Say that a simplicial set map $p: X \to Y$ is a {\it right fibration}
if $p$ has the right lifting property with respect to all inclusions
$\Lambda^{n}_{k} \subset \Delta^{n}$ with $k > 0$. This definition is
consistent with \cite{Lurie-HTT}.

\begin{lemma}\label{lem 13}
Suppose that $p: X \to Y$ is a right fibration, and that $X$ and $Y$
are quasi-categories. Suppose that $\alpha: \Delta^{1} \to X$ is a
$1$-simplex of $X$ such that $p(\alpha)$ is a quasi-isomorphism. Then
$\alpha$ is a quasi-isomorphism.
\end{lemma}

We also say that the map $p$ of Lemma \ref{lem 13} {\it creates}
quasi-isomorphisms.

\begin{proof}
The simplex $p(\alpha)$ has a right inverse, so there is a $2$ simplex
$\sigma: \Delta^{2} \to Y$ with boundary $\partial \sigma =
(p(\alpha),p(y),\gamma)$. The lifting exists in the diagram
\begin{equation*}
\xymatrix{
\Lambda^{2}_{2} \ar[r]^{(\alpha,y,\ )} \ar[d] & X \ar[d]^{p} \\
\Delta^{2} \ar[r]_{\sigma} \ar@{.>}[ur]^{\theta} & Y
}
\end{equation*}
so that $\alpha$ has a right inverse $\zeta=d_{2}\theta$ in
$P(X)$. Similarly $\zeta$ has a right inverse $\omega$ in $P(X)$. Thus,
\begin{equation*}
\alpha = \alpha \cdot \zeta \cdot \omega = \omega
\end{equation*}
in $P(X)$, so that $\alpha$ is a quasi-isomorphism.
\end{proof}

Write $\Delta^{m} \ast\Delta^{n} \cong \Delta^{m+n+1}$ for the (poset)
join of the simplices $\Delta^{n}$ and $\Delta^{m}$. The join $X \ast
Y$ of the simplicial sets $X$ and $Y$ is defined by the colimit formula
\begin{equation*}
X \ast Y = \varinjlim_{\Delta^{m} \to X,\ \Delta^{n} \to
  Y}\ \Delta^{m} \ast \Delta^{n}.
\end{equation*}
where the colimit is computed over the product $\mathbf{\Delta}/X
\times \mathbf{\Delta}/Y$ of the respective simplex categories.
It is relatively easy to show that the maps
\begin{equation*}
\begin{aligned}
&(\Lambda^{m}_{k} \ast \Delta^{n}) \cup (\Delta^{m} \ast \partial\Delta^{n}) \to \Delta^{m} \ast \Delta^{n}\\
&(\partial\Delta^{m} \ast \Delta^{n}) \cup (\Delta^{m} \ast \partial^{n}_{k}) \to \Delta^{m} \ast \Delta^{n}\\
&(\partial\Delta^{m} \ast \Delta^{n}) \cup (\Delta^{m} \ast \partial\Delta^{n}) \to
\Delta^{m} \ast \Delta^{n}
\end{aligned}
\end{equation*}
are isomorphic to the maps $\Lambda^{m+n+1}_{k} \subset
\Delta^{m+n+1}$, $\Lambda^{m+n+1}_{m+k+1} \subset \Delta^{m+n+1}$, and
$\partial\Delta^{m+n+1} \subset \Delta^{m+n+1}$, respectively. 

Suppose that $\mathbf{hom}_{j}(X,Y)$ is the simplicial set with
$r$-simplices given by the maps $\Delta^{r} \ast X \to Y$. If $p: X
\to Y$ is an inner fibration then the induced map
\begin{equation*}
(i^{\ast},p_{\ast}): 
\mathbf{hom}_{j}(\Delta^{n},X) \to \mathbf{hom}_{j}(\partial\Delta^{n},X) 
\times_{\mathbf{hom}_{j}(\partial\Delta^{n},Y)} \mathbf{hom}_{j}(\Delta^{n},Y)
\end{equation*}
is a right fibration provided that $n \geq 0$. It follows (by setting
$Y = \ast$) that all maps
\begin{equation*}
\mathbf{hom}_{j}(\Delta^{n},X) \to \mathbf{hom}_{j}(\partial\Delta^{n},X)
\end{equation*}
are right fibrations if $X$ is a quasi-category. 

\begin{proposition}
Suppose given a commutative solid arrow diagram
\begin{equation*}
\xymatrix{
\Lambda^{n}_{0} \ar[r]^{\alpha} \ar[d] & X \ar[d]^{p} \\
\Delta^{n} \ar[r] \ar@{.>}[ur] & Y
}
\end{equation*}
where $n \geq 2$, $X$ and $Y$ are quasi-categories, $p$ is an inner fibration,
and the map $\alpha$ takes the $1$-simplex $0 \to 1$ to a
quasi-isomorphism of $X$. Then the dotted arrow
lifting exists, making the diagram commute.
\end{proposition}

\begin{proof}
The lifting problem
\begin{equation*}
\xymatrix{
\Lambda^{n}_{0} \ar[r]^{\alpha} \ar[d] & X \ar[d]^{p} \\
\Delta^{n} \ar[r] \ar@{.>}[ur] & Y
}
\end{equation*}
is isomorphic to a lifting problem
\begin{equation*}
\xymatrix{
\Lambda^{1}_{0} \ar[r]^{x} \ar[d] & \mathbf{hom}_{j}(\Delta^{n-2},X) \ar[d]^{(i^{\ast},p_{\ast})} \\
\Delta^{1} \ar[r]_-{\alpha_{\ast}} \ar@{.>}[ur] & 
\mathbf{hom}_{j}(\partial\Delta^{n-2},X) \times_{\mathbf{hom}_{j}(\partial\Delta^{n-2},Y)} \mathbf{hom}_{j}(\Delta^{n-2},Y)
}
\end{equation*}
where $i: \partial\Delta^{n-2} \subset \Delta^{n}$ The map $(i^{\ast},p_{\ast})$ is a right fibration, as is the projection map
\begin{equation*}
\mathbf{hom}_{j}(\partial\Delta^{n-2},X) \times_{\mathbf{hom}_{j}(\partial\Delta^{n-2},Y)} \mathbf{hom}_{j}(\Delta^{n-2},Y)
\to \mathbf{hom}_{j}(\partial\Delta^{n-2},X).
\end{equation*}
The inclusion of any vertex of $\partial\Delta^{n-2}$ induces a right fibration
\begin{equation*}
\mathbf{hom}_{j}(\partial\Delta^{n-2},X) \to X.
\end{equation*}
It follows from Lemma \ref{lem 13} that the $1$-simplex $\alpha_{\ast}$ is
a quasi-isomorphism.

Let $\beta$ be an inverse for $\alpha_{\ast}$. The lifting problem 
\begin{equation*}
\xymatrix{
\Lambda^{1}_{1} \ar[r]^{x} \ar[d] & \mathbf{hom}_{j}(\Delta^{n-2},X) \ar[d]^{(i^{\ast},p_{\ast})} \\
\Delta^{1} \ar[r]_-{\beta} \ar@{.>}[ur]^{\theta} & 
\mathbf{hom}_{j}(\partial\Delta^{n-2},X) \times_{\mathbf{hom}_{j}(\partial\Delta^{n-2},Y)} \mathbf{hom}_{j}(\Delta^{n-2},Y)
}
\end{equation*}
has a solution since the map $(i^{\ast},p_{\ast})$ is a right
fibration.  It follows from Lemma \ref{lem 13} that $\theta$ is a
quasi-isomorphism of $\mathbf{hom}_{j}(\Delta^{n-2},X)$. Let $\gamma$
be an inverse of $\theta$. Then $i^{\ast}(\gamma) \sim \alpha_{\ast}$ so that
$\alpha$ lifts along $i^{\ast}$.
\end{proof}

Here's the dual statement:

\begin{corollary}
Suppose given a commutative solid arrow diagram
\begin{equation*}
\xymatrix{
\Lambda^{n}_{n} \ar[r]^{\alpha} \ar[d] & X \ar[d]^{p} \\
\Delta^{n} \ar[r] \ar@{.>}[ur] & Y
}
\end{equation*}
where $n \geq 2$, $X$ and $Y$ are quasi-isomorphisms, $p$ is an inner fibration,
and the map $\alpha$ takes the $1$-simplex $n-1 \to n$ to a
quasi-isomorphism of $X$. Then the dotted arrow
lifting exists, making the diagram commute.
\end{corollary}

\begin{corollary}\label{cor 16}
Suppose that $X$ is a quasi-category and suppose that $n \geq 2$. Then
the liftings exist in the diagrams
\begin{equation*}
{\xymatrix{
\Lambda^{n}_{0} \ar[r]^{\alpha} \ar[d] & X \\
\Delta^{n} \ar@{.>}[ur]
}
}\quad
{\xymatrix{
\Lambda^{n}_{n} \ar[r]^{\beta} \ar[d] & X \\
\Delta^{n} \ar@{.>}[ur]
}
}
\end{equation*}
provided that $\alpha$ takes the simplex $0 \to 1$ to a
quasi-isomorphism of $X$, respectively that $\beta$ takes the simplex $n-1
\to n$ to a quasi-isomorphism of $X$.
\end{corollary}

\begin{corollary}\label{cor 17}
Suppose that $X$ is a quasi-category. Then we have the following:
\begin{itemize}
\item[1)] 
If $P(X)$ is a groupoid, then $X$ is a Kan complex. Thus, a
  quasi-category $X$ is a Kan complex if and only if $P(X)$ is a
  groupoid.
\item[2)]
Let $J(X) \subset X$ be the subobject consisting of those simplices 
$\sigma: \Delta^{n} \to X$ such that all composites
\begin{equation*}
\Delta^{1} \to \Delta^{n} \xrightarrow{\sigma} X
\end{equation*}
represent isomorphisms of $P(X)$. Then $J(X)$ is a Kan complex, and 
it is the maximal Kan subcomplex of $X$.
\end{itemize}
\end{corollary}

The subcomplex $J(X)$ is often called the {\it core} of the
quasi-category $X$.

\begin{corollary}\label{cor 18}
Suppose that $p: X \to Y$ is an inner fibration, where $X$ and $Y$ are
Kan complexes. Suppose also that $p$ has the path lifting
property. Then $p$ is a Kan fibration.
\end{corollary}

\begin{proof}
The lifting problems
\begin{equation*}
{\xymatrix{
\Lambda^{n}_{0} \ar[r] \ar[d] & X \ar[d]^{p} \\
\Delta^{n} \ar[r] \ar@{.>}[ur] & Y
}
}\quad
{\xymatrix{
\Lambda^{n}_{n} \ar[r] \ar[d] & X \ar[d]^{p} \\
\Delta^{n} \ar[r] \ar@{.>}[ur] & Y
}
}
\end{equation*}
have solutions for $n=1$ by assumption, and have solutions for $n>1$
by the results above, since every $1$-simplex of a Kan complex is a
quasi-isomorphism.
\end{proof}

Suppose that $X$ is a simplicial set, and write
$\mathbf{hom}_{I}(\Delta^{1},X)$ for the simplicial set whose
$n$-simplices are the maps $\alpha: \Delta^{1} \times \Delta^{n} \to
X$ such that all vertices $x$ of $\Delta^{n}$ determine composites
\begin{equation*}
\Delta^{1} \xrightarrow{(1,x)} \Delta^{1} \times \Delta^{n} \xrightarrow{\alpha} X,
\end{equation*}
which represent isomorphisms of the path category $P(X)$. Write $I =
B\pi(\Delta^{1}))$, and observe that $\pi(\Delta^{1})$ is the free
groupoid on the ordinal number $\mathbf{1}$, and is therefore the
groupoid that is freely generated by one non-trivial arrow $\eta: 0
\to 1$. The path
\begin{equation*}
\Delta^{1} \xrightarrow{\eta} I
\end{equation*}
induces a map
\begin{equation*}
\mathbf{hom}(I,X) \xrightarrow{\eta^{\ast}} \mathbf{hom}_{I}(\Delta^{1},X).
\end{equation*}

\begin{lemma}\label{lem 19}
Suppose that $X$ is a quasi-category and let $n > 0$. Then the inner fibration
\begin{equation*}
i^{\ast}: \mathbf{hom}(\Delta^{n},X) \to \mathbf{hom}(\partial\Delta^{n},X)
\end{equation*}
induced by the inclusion $i: \partial\Delta^{n} \subset \Delta^{n}$
creates quasi-isomorphisms.
\end{lemma}

\begin{remark}
Lemma \ref{lem 19} implies that the $n$-simplices of
$\mathbf{hom}_{I}(\Delta^{1},X)$ for a quasi-category $X$ are exactly
the quasi isomorphisms $\Delta^{n} \times \Delta^{1} \to X$ of the
quasi-category $\mathbf{hom}(\Delta^{n},X)$, provided that $X$ is a
quasi-category.
\end{remark}

%%%%%%%%%%%%%%%%%%%

\begin{proof}[Proof of Lemma \ref{lem 19}]
Suppose that $\alpha: \Delta^{n} \times \Delta^{1} \to X$ is a morphism such that the composite
\begin{equation*}
\alpha_{\ast}: \partial\Delta^{n} \times \Delta^{1} \xrightarrow{i \times 1} \Delta^{n} \times \Delta^{1} \xrightarrow{\alpha} X
\end{equation*}
is a quasi-isomorphism of $\mathbf{hom}(\partial\Delta^{n},X)$.
Then there is a $1$-simplex $\beta: \partial\Delta^{n} \times
\Delta^{1} \to X$, and a $2$-simplex $\sigma: \partial\Delta^{n}
\times \Delta^{2} \to X$ of $\mathbf{hom}(\partial\Delta^{n},X)$ such that
\begin{equation*}
\partial(\sigma) = (d_{0}\sigma,d_{1}\sigma,d_{2}\sigma) 
= (\beta,s_{0}(d_{1}\alpha_{\ast}),\alpha_{\ast}).
\end{equation*}
We therefore have an induced map
\begin{equation*}
(\partial\Delta^{n} \times \Delta^{2}) \cup (\Delta^{n} \times \Lambda^{2}_{0}) 
\xrightarrow{(\sigma,(\ ,s_{0}(d_{1}\alpha),\alpha))} X
\end{equation*}
and I claim that it suffices to show that this map extends to a morphism 
$\theta: \Delta^{n} \times \Delta^{2} \to X$. 

If so, then $\beta' = d_{1}\sigma$ is a left inverse for $\alpha$ in
$P(\mathbf{hom}(\Delta^{n},X)$, and so every simplex $\alpha$ such
that $i^{\ast}(\alpha)$ is a quasi-isomorphism has a left inverse in
the path category, and is therefore monic. But then $\beta'$ is also
monic in the path category, so that
\begin{equation*}
\beta' \cdot \alpha \cdot \beta' = \beta'
\end{equation*}
forces $\alpha \cdot \beta' = 1$, so that $\alpha$ is a quasi-isomorphism.

We must therefore solve the extension problem
\begin{equation*}
\xymatrix@C=60pt{
(\partial\Delta^{n} \times \Delta^{2}) \cup (\Delta^{n} \times \Lambda^{2}_{0}) 
\ar[r]^-{(\sigma,(\ ,s_{0}(d_{1}\alpha),\alpha))} \ar[d]  & X \\
\Delta^{n} \times \Delta^{2} \ar@{.>}[ur]
}
\end{equation*}
We do this by using the ordering on the set of non-degenerate
simplices of $\Delta^{n} \times \Delta^{2}$ of dimension $n+2$ that
is developed in connection with the proof of Theorem \ref{th 45}
below.

A non-degenerate $(n+2)$-simplex
\begin{equation*}
(0,0) \to \dots \to (n,2)
\end{equation*}
is a path in the poset $\mathbf{n} \times \mathbf{2}$ with first
segment either $(0,0) \to (1,0)$ or $(0,0) \to (0,1)$. Let $S_{0}$ be the
set of non-degenerate $(n+2)$-simplices starting with $(0,0) \to
(1,0)$ and let $S'$ be those non-degenerate $(n+2)$-simplices starting
with $(0,0) \to (0,1)$. The path
\begin{equation*} 
P:\quad 
{\xymatrix{
& (1,2) \ar[r] & \dots \ar[r] & (n,2) \\ 
& (1,1) \ar[u] \\
(0,0) \ar[r] & (1,0) \ar[u] 
}}
\end{equation*}
is the minimal simplex of $S_{0}$, and the path
\begin{equation*}
Q_{n}:\quad
{\xymatrix{
&& (n,2) \\
(0,1) \ar[r] & \dots \ar[r] & (n,1) \ar[u] \\
(0,0) \ar[u]
}}
\end{equation*}
is the maximal simplex of $S'$.

If $T$ is a set of non-degenerate $(n+2)$-simplices, write
\begin{equation*}
(\Delta^{n} \times \Delta^{2})^{(T)} 
\end{equation*}
for the subcomplex of $\Delta^{n} \times \Delta^{2}$ that is generated
by the subcomplex
\begin{equation*}
(\partial\Delta^{n} \times \Delta^{2}) \cup (\Delta^{n} \times \Lambda^{2}_{0})
\end{equation*}
and the members of $T$.

All members $P'$ of $S_{0}$ have $d_{0}P'$ and $d_{n+2}P'$ in the subcomplex 
\begin{equation*}
(\Delta^{n} \times \Delta^{2})^{(\emptyset)} =  (\partial\Delta^{n} \times \Delta^{2}) \cup (\Delta^{n} \times \Lambda^{2}_{0}),
\end{equation*}
so the inclusion
\begin{equation*}
(\Delta^{n} \times \Delta^{2})^{(\emptyset)} \subset (\Delta^{n} \times \Delta^{2})^{(S_{0})}
\end{equation*}
is inner anodyne by Lemma \ref{lem 8}.

Suppose that $Q_{i}$ is the $(n+2)$-simplex
\begin{equation*}
\xymatrix{
&& (i,2) \ar[r] & \dots \ar[r] & (n,2) \\
(0,1) \ar[r] & \dots \ar[r] & (i,1) \ar[u] \\
(0,0) \ar[u] 
}
\end{equation*}
and write 
\begin{equation*}
S'_{i} = S \cup \{ Q_{0}, \dots , Q_{i} \}.
\end{equation*}

Suppose first that $i < n$.

The face $d_{0}Q_{i}$ has faces $d_{0}d_{0}Q_{i}$ and
$d_{n+1}d_{0}Q_{i}$ in the the image of the boundary subcomplex
$\partial (\Delta^{n} \times \Delta^{1})$ under the map $1 \times
d^{0}: \Delta^{n} \times \Delta^{1} \to \Delta^{n} \times \Delta^{2}$,
and the intersection of the boundary of $d_{0}Q_{i}$ with $(\Delta^{n}
\times \Delta^{2})^{(S'_{i-1})}$ is missing the face
$d_{i+1}d_{0}Q_{i}$.

If $i < n$, then $d_{i+2}Q_{i}$ is an interior simplex and it is only
a face of $Q_{i}$ and $Q_{i+1}$, so that $d_{i+2}Q_{i}$ is not in the
subcomplex
\begin{equation*}
(\Delta^{n} \times \Delta^{2})^{(S'_{i-1})}.
\end{equation*}

It follows that the inclusion
\begin{equation*}
(\Delta^{n} \times \Delta^{2})^{(S'_{i-1})} \subset 
(\Delta^{n} \times \Delta^{2})^{(S'_{i})}
\end{equation*}
is a sequence of inner anodyne extensions that is achieved by first attaching
$d_{0}Q_{i}$ and then attaching $Q_{i}$ if $i<n$. 

Thus, there is a sequence of inner anodyne extensions
\begin{equation*}
(\Delta^{n} \times \Delta^{2})^{(\emptyset)} \subset (\Delta^{n} \times \Delta^{2})^{(S_{0})} \subset (\Delta^{n} \times \Delta^{2})^{(S'_{i})}
\end{equation*}
if $i < n$. The ``last'' simplex $Q_{n}$ has 
\begin{equation*}
d_{i}Q_{n} \in (\Delta^{n} \times \Delta^{2})^{(S'_{n-1})}
\end{equation*}
if $i > 0$, and $d_{0}Q_{n}$ is not a face any simplex in $S'_{n-1}$
since it is a maximal non-degenerate simplex of $\Delta^{n} \times
\Delta^{1}$.

The image of the $1$-simplex $(0,0) \to (0,1)$ under
the composite
\begin{equation*}
\alpha_{\ast}: \partial\Delta^{n} \times \Delta^{1} \xrightarrow{i \times 1} \Delta^{n} \times \Delta^{1} \xrightarrow{\alpha} X
\end{equation*}
is a quasi-isomorphism of $X$, since that composite is a
quasi-isomorphism of $\mathbf{hom}(\partial\Delta^{n},X)$. It follows
from Corollary \ref{cor 16} that the last of the extension problems
(the dotted arrow) in the list
\begin{equation*}
\xymatrix{
(\Delta^{n} \times \Delta^{2})^{(\emptyset)} \ar[d] \ar[r] & X \\
(\Delta^{n} \times \Delta^{2})^{(S)} \ar[d]  \ar[ur] \\
(\Delta^{n} \times \Delta^{2})^{(S'_{n-1})} \ar[d] \ar[uur] \\ 
\Delta^{n} \times \Delta^{2} \ar@{.>}[uuur]
}
\end{equation*}
can be solved.
\end{proof}

\begin{corollary}
Suppose that $i: A \to B$ is a cofibration of simplicial sets such
that $i$ is a bijection on vertices. Suppose that $X$ is a
quasi-category. Then the induced map
\begin{equation*}
i^{\ast}: \mathbf{hom}(B,X) \to \mathbf{hom}(A,X)
\end{equation*}
of quasi-categories creates quasi-isomorphisms.
\end{corollary}

\begin{lemma}\label{lem 22}
Suppose that $X$ is a quasi-category, and suppose given
a diagram
\begin{equation*}
\xymatrix{
(\Delta^{n} \times \{ \epsilon \}) \cup (\partial\Delta^{n} \times \Delta^{1}) 
\ar[r]^-{(\beta,\alpha)} \ar[d] & X \\
\Delta^{n} \times \Delta^{1} \ar@{.>}[ur]
}
\end{equation*}
where $\alpha: \partial\Delta^{n} \times \Delta^{1} \to X$ is a
quasi-isomorphism of 
$\mathbf{hom}(\partial\Delta^{n},X)$, and $\epsilon$ is $0$ or
$1$. Then the indicated extension problem can be solved.
\end{lemma}

\begin{proof}
We'll suppose that $\epsilon = 0$. The case $\epsilon=1$ is similar
(or even dual).

Write $h_{i}$ for the non-degenerate $(n+1)$-simplex
\begin{equation*}
\xymatrix{
(0,0) \ar[r] & \dots \ar[r] & (i,0) \ar[d] \\
&&(i,1) \ar[r] & \dots \ar[r] & (n,1)
}
\end{equation*}
Let $T = \{ h_{1}, \dots ,h_{n} \}$ and write 
$(\Delta^{n} \times \Delta^{1})^{(T)}$
for the subcomplex of $\Delta^{n} \times \Delta^{1}$ generated by the
subcomplex 
\begin{equation*}
(\Delta^{n} \times \Delta^{1})^{(\emptyset)} =
(\Delta^{n} \times \{ 0 \}) \cup (\partial\Delta^{n}
\times \Delta^{1})
\end{equation*} 
and the simplices in $T$. Then the inclusion 
\begin{equation*}
(\Delta^{n} \times \Delta^{1})^{(\emptyset)} \subset (\Delta^{n} \times \Delta^{1})^{(S_{n-1})}
\end{equation*}
is inner anodyne and there is a pushout diagram
\begin{equation*}
\xymatrix{
\Lambda^{n+1}_{0} \ar[r] \ar[d] 
& (\Delta^{n} \times \Delta^{1})^{(T)} \ar[d] \\
\Delta^{n+1} \ar[r]_-{h_{0}} & \Delta^{n} \times \Delta^{1} 
}
\end{equation*}
In the diagram
\begin{equation*}
\xymatrix{
(\Delta^{n} \times \Delta^{1})^{(\emptyset)} \ar[r]^-{(\beta,\alpha)} \ar[d] 
& X \\
(\Delta^{n} \times \Delta^{1})^{(T)} \ar[ur]^{\gamma} \ar[d] \\
\Delta^{n} \times \Delta^{1} \ar@{.>}[uur]
}
\end{equation*}
the extension $\gamma$ takes the simplex $(0,0) \to (0,1)$ to a
quasi-isomorphism of $X$, so the indicated extension problem can be
solved by Corollary \ref{cor 16}.
\end{proof}

\begin{corollary}\label{cor 23}
Suppose that $X$ is a quasi-category. Then the map
\begin{equation*}
i^{\ast}: J(\mathbf{hom}(\Delta^{n},X)) \to J(\mathbf{hom}(\partial\Delta^{n},X))
\end{equation*}
is a Kan fibration for $n \geq 0$.
\end{corollary}

\begin{proof}
The map $i^{\ast}: \mathbf{hom}(\Delta^{n},X) \to
\mathbf{hom}(\partial\Delta^{n},X)$ is an inner fibration by Theorem
\ref{thm 9}, and has the path lifting property for
quasi-isomorphisms by Lemma \ref{lem 22}. 

Lemma \ref{lem 19} says that $i^{\ast}$ creates quasi-isomorphisms. It
follows that the diagram
\begin{equation*}
\xymatrix{
J(\mathbf{hom}(\Delta^{n},X)) \ar[r] \ar[d]_{i^{\ast}} 
& \mathbf{hom}(\Delta^{n},X) \ar[d]^{i^{\ast}} \\
J(\mathbf{hom}(\partial\Delta^{n},X)) \ar[r]
& \mathbf{hom}(\partial\Delta^{n},X)
}
\end{equation*}
is a pullback. The map
\begin{equation*}
i^{\ast}: J(\mathbf{hom}(\Delta^{n},X)) \to J(\mathbf{hom}(\Delta^{n},X))
\end{equation*}
is thus an inner fibration between Kan complexes that has the path lifting
property, and is therefore a Kan fibration by Corollary \ref{cor 18}.
\end{proof}

\begin{proposition}\label{prop 24}
Suppose that $X$ is a quasi-category. Then the map
\begin{equation*}
\eta^{\ast}: \mathbf{hom}(I,X) \to \mathbf{hom}_{I}(\Delta^{1},X)
\end{equation*}
is a trivial Kan fibration.
\end{proposition}

\begin{proof}
The lifting problem
\begin{equation*}
\xymatrix{
\partial\Delta^{n} \ar[r] \ar[d] & \mathbf{hom}(I,X) \ar[d]^{\eta^{\ast}} \\
\Delta^{n} \ar[r]_-{\alpha} \ar@{.>}[ur] & \mathbf{hom}_{I}(\Delta^{1},X)
}
\end{equation*}
is isomorphic to the lifting problem
\begin{equation}\label{eq 1}
\xymatrix{
\Delta^{1} \ar[r]^-{\alpha_{\ast}} \ar[d]_{\eta} & \mathbf{hom}(\Delta^{n},X) \ar[d]^{i^{\ast}} \\
I \ar[r] \ar@{.>}[ur] & \mathbf{hom}(\partial\Delta^{n},X)
}
\end{equation}
The map $\alpha_{\ast}$ is a quasi-isomorphism by Lemma \ref{lem 19},
and so the diagram (\ref{eq 1}) factors through the diagram
\begin{equation*}
\xymatrix{
\Delta^{1} \ar[r]^-{\alpha_{\ast}} \ar[d]_{\eta} & J(\mathbf{hom}(\Delta^{n},X)) \ar[d]^{i^{\ast}} \\
I \ar[r] \ar@{.>}[ur] & J(\mathbf{hom}(\partial\Delta^{n},X))
}
\end{equation*}
The map $i^{\ast}$ is a Kan fibration by Corollary \ref{cor 23}.  The
map $\eta$ is a trivial cofibration in the standard model structure
for simplicial sets, so the lifting exists.
\end{proof}

\section{The quasi-category model structure}

Let $I = B(\pi(\Delta^{1}))$ represent an interval theory on simplicial
sets in the sense of \cite{J40}. In particular, recall that the assignment 
\begin{equation*}
\mathcal{P}(\underline{n}) \mapsto
\square^{n} = I^{\times n}
\end{equation*}
defines a functor $I^{\bullet}: \mathbf{\square} \to s\mathbf{Set}$, giving a
representation of the box category $\mathbf{\square}$ in simplicial
sets.

The existence of two distinct maps $0,1: \Delta^{0} \to I$ is part of
the structure of an interval theory, and in this case these maps are
given by the the two objects $0,1$ of the groupoid
$\pi(\Delta^{1})$. An $I$-homotopy $f \sim_{I} g$ for an
interval $I$ is a commutative diagram
\begin{equation*}
\xymatrix@R=12pt{
K  \ar[d]_{(1_{K},0)} \ar[dr]^{f} \\
K \times I \ar[r]^{h} & X \\
K \ar[u]^{(1_{K},1)} \ar[ur]_{g}
}
\end{equation*}
Write $\pi_{I}(K,X)$ for the corresponding set of $I$-homotopy
classes of maps.

Following \cite{J40}, write $\partial
\square^{n}$ for the union of the faces 
\begin{equation*}
I^{\times (k-1)} \times \{
\epsilon \} \times I^{\times (n-k)}
\end{equation*}
of $\square^{n}$ in the category of simplicial sets. The subcomplex
$\sqcap^{n}_{(k,\epsilon)} \subset \partial \square^{n}$ is the result
of deleting the face $I^{\times (k-1)} \times \{ \epsilon \} \times
I^{\times (n-k)}$ from $\partial\square^{n}$.

The set $S$ of inner anodyne extensions $\Lambda^{n}_{k} \subset
\Delta^{n}$ and the interval $I$ together determine an $(I,S)$-model
structure on the category $s\mathbf{Set}$ of simplicial sets
\cite{J40}, for which the cofibrations are the monomorphisms, and the
fibrant objects are those simplicial sets $Z$ for which the map $Z \to
\ast$ has the right lifting property with respect all inclusions
\begin{equation}\label{eq 2}
(\Lambda^{n}_{k} \times \square^{m}) \cup (\Delta^{n}) \times \partial\square^{m}) \subset \Delta^{n} \times \square^{m}
\end{equation}
induced by inner horns $\Lambda^{n}_{k} \subset \Delta^{n}$, and with
respect to all maps
\begin{equation}\label{eq 3}
(\partial\Delta^{n} \times \square^{m}) \cup (\Delta^{n} \times \sqcap^{m}_{(k,\epsilon)}) \subset \Delta^{n} \times \square^{m}.
\end{equation}
A weak equivalence of the $(I,S)$-model structure is a map $X \to Y$ which induces an isomorphism
\begin{equation*}
\pi_{I}(Y,Z) \xrightarrow{\cong} \pi_{I}(X,Z)
\end{equation*} 
in $I$-homotopy classes of maps for all fibrant objects
$Z$. 

There is a
natural cofibration $j: X \to LX$, such that the map $j$ is in the
saturation of the set of maps described in (\ref{eq 2}) and 
(\ref{eq 3}), and $LX$ has the right lifting property with respect to all
such maps, and is therefore fibrant for the $(I,S)$-model
structure. The map $j: X \to LX$ is a fibrant model for the
$(I,S)$-model structure.

A map $X \to Y$ is a weak equivalence if and only if the
induced map $LX \to LY$ is an $I$-homotopy equivalence. 
\medskip

If $h: K \times I \to X$ is an $I$-homotopy taking values in a
quasi-category $X$ then the composite map
\begin{equation*}
K \times \Delta^{1} \xrightarrow{1_{K} \times \eta} K \times I \xrightarrow{h} X
\end{equation*}
is a quasi-isomorphism of quasi-category $\mathbf{hom}(K,X)$: the
requisite inverse and $2$-simplices are defined in $I$. Following
Joyal \cite{Joyal-quasi-cat}, let $\tau_{0}(K,X)$ denote the set of isomorphism
classes in $P(\mathbf{hom}(K,X))$. It follows that there is an induced
function
\begin{equation*}
\pi_{I}(K,X) \xrightarrow{\eta^{\ast}} \tau_{0}(K,X).
\end{equation*}
We have, from Proposition \ref{prop 24}, a trivial Kan fibration
\begin{equation*}
\eta^{\ast}: \mathbf{hom}(I,\mathbf{hom}(K,X)) \to \mathbf{hom}_{I}(\Delta^{1},\mathbf{hom}(K,X)).
\end{equation*}
The vertices of the space
$\mathbf{hom}_{I}(\Delta^{1},\mathbf{hom}(K,X))$ are the
quasi-isomor\-phisms of the quasi-category $\mathbf{hom}(K,X)$. The
trivial fibration $\eta^{\ast}$ is, among other things, surjective on
vertices, and so for every quasi-isomorphism $h: K \times \Delta^{1}
\to X$ there is an extension
\begin{equation*}
\xymatrix{
K \times \Delta^{1} \ar[r]^-{h} \ar[d]_{1_{K} \times \eta} & X \\
K \times I \ar@{.>}[ur]_{H}
}
\end{equation*}
Thus if there is a homotopy $h: K \times \Delta^{1} \to X$ from $f$ to
$g$ that is a quasi-isomorphism of $\mathbf{hom}(K,X)$, then there is
an $I$-homotopy $H: K \times I \to X$ from $f$ to $g$, and
conversely. We have proven the following:

\begin{proposition}\label{prop 25}
Suppose that $X$ is a quasi-category and that $K$ is a simplicial
set. Then precomposing with the map $\eta: \Delta^{1} \to I$ defines a
bijection
\begin{equation*}
\pi_{I}(K,X) \cong \tau_{0}(K,X).
\end{equation*}
This bijection is natural in simplicial sets $K$ and quasi-categories $X$.
\end{proposition}

\begin{lemma}\label{lem 26}
Suppose that $X$ is a quasi-category. Then all extension problems
\begin{equation*}
\xymatrix{
(\partial\Delta^{m} \times \square^{n}) \cup (\Delta^{m} \times \sqcap^{n}_{(i,\epsilon)}) \ar[r] \ar[d] & X \\
\Delta^{m} \times \square^{n} \ar@{.>}[ur]
}
\end{equation*}
can be solved for $\epsilon = 0,1$. 
\end{lemma}

\begin{proof}
The extension problem in the statement of the Lemma can be rewritten
as a lifting problem
\begin{equation}\label{eq 4}
\xymatrix{
\sqcap^{n}_{(i,\epsilon)} \ar[r]^-{\alpha} \ar[d]_{j} & \mathbf{hom}(\Delta^{m},X) \ar[d]^{i^{\ast}} \\
\square^{n} \ar[r]_-{\beta} \ar@{.>}[ur] & \mathbf{hom}(\partial\Delta^{m},X)
}
\end{equation}
by adjointness. The maps $\alpha$ and $\beta$ both map $1$-simplices
to quasi-isomorphisms, so the lifting problem (\ref{eq 4}) factors
through a lifting problem
\begin{equation*}
\xymatrix{
\sqcap^{n}_{(i,\epsilon)} \ar[r]^-{\alpha} \ar[d]_{j} & J(\mathbf{hom}(\Delta^{m},X)) \ar[d]^{i^{\ast}} \\
\square^{n} \ar[r]_-{\beta} \ar@{.>}[ur] & J(\mathbf{hom}(\partial\Delta^{m},X))
}
\end{equation*}
The map $i^{\ast}$ in the diagram is a Kan fibration by Corollary \ref{cor 23},
 and the inclusion $\sqcap^{n}_{(i,\epsilon)} \subset
\square^{n}$ is a standard weak equivalence (of contractible spaces),
so that the dotted arrow exists.
\end{proof}

\begin{theorem}\label{thm 27}
Quasi-categories are the fibrant objects for the $(I,S)$-model
structure on simplicial sets that is determined by the set $S$ of
inner anodyne extensions and the interval theory defined by the space
$I=B(\pi(\Delta^{1}))$.
\end{theorem}

\begin{proof}
If the inclusion $A \subset B$ is inner anodyne, then so are all inclusions
\begin{equation*}
(A \times \square^{n}) \cup (B \times \partial\square^{n}) \subset (B \times \square^{n}).
\end{equation*}
This is a consequence of Theorem \ref{thm 9}. 

If $X$ is a
quasi-category, then the map $X \to \ast$ has the right lifting
property with respect to all inclusions
\begin{equation*}
(\partial\Delta^{m} \times \square^{n}) \cup (\Delta^{m} \times \sqcap^{n}_{(i,\epsilon)})  \subset \Delta^{m} \times \square^{n}
\end{equation*}
by Lemma \ref{lem 26}. 
\end{proof}

Proposition \ref{prop 25}  says that the naive homotopy
classes of maps $\pi_{I}(K,X)$ taking values in a quasi-category $X$ for
the $(I,S)$-model structure for simplicial sets coincide up to natural
bijection with Joyal's set 
\begin{equation*}
\tau_{0}(K,X) = \pi_{0}J(\mathbf{hom}(K,X)),
\end{equation*} 
so that the fibrant objects, weak equivalences and cofibrations of the
$(I,S)$-structure coincide with the respective classes of maps in
Joyal's model structure for quasi-categories \cite{Joyal-quasi-cat}.
It follows that the $(I,S)$-structure for simplicial sets coincides
with Joyal's structure. 
\medskip

In particular, the weak equivalences and the fibrations for the
$(I,S)$-struc\-ture are the {\it weak categorical equivalences} and the
{\it pseudo-fibrations} of \cite{Joyal-quasi-cat}, respectively. We
shall continue to use these terms.

In particular, a weak categorical equivalence is a
simplicial set map $f: X \to Y$ such that the induced map
\begin{equation*}
f^{\ast}: \pi_{I}(Y,Z) \to \pi_{I}(X,Z)
\end{equation*}
is a bijection for all quasi-categories $Z$. 

Observe that a map $p: X \to Y$ is a weak categorical equivalence and
a pseudo-fibration if and only if it is a trivial fibration in the
standard model structure for simplicial sets.

The natural fibrant model $j_{X}: X \to LX$ for the $(I,S)$-model
structure also has a much simpler construction, such that $j$ is an
inner anodyne extension and $LX$ is a quasi-category.

%%%%%%%%%%%%%%%%%%%%%%%%

\begin{lemma}\label{lem 28}
If $g: X \to Y$ is a categorical weak equivalence, then the induced map
$g_{\ast}: P(X) \to P(Y)$ is an equivalence of categories. 
\end{lemma}

\begin{proof}
There is a commutative diagram
\begin{equation*}
\xymatrix{
X \ar[r]^{j_{X}} \ar[d]_{g} & LX \ar[d]^{g_{\ast}} \\
Y \ar[r]_{j_{Y}}& LY
}
\end{equation*}
where $j_{X}$ and $j_{Y}$ are inner anodyne extensions and $LX$ and
$LY$ are quasi-categories. The map $g_{\ast}$ is a categorical weak
equivalence of quasi-categories, and is therefore an $I$-homotopy
equivalence.  It follows that the induced map
\begin{equation*}
g_{\ast}: P(LX) \to P(LY)
\end{equation*}
is an equivalence of categories. 

In effect, the functor $X \mapsto P(X)$ preserves finite products and
$P(I) \cong \pi(\Delta^{1})$. Finally, Lemma \ref{lem 6} implies that
the inner anodyne extensions $i_{X},i_{Y}$ induce isomorphisms $P(X)
\cong P(X_{f})$ and $P(Y) \cong P(Y_{f})$.
\end{proof}

We also have the following, by a very similar argument:

\begin{lemma}\label{lem 29}
Every categorical weak equivalence is a standard weak equivalence
of simplicial sets.
\end{lemma}

\begin{proof}
The inner horn inclusions $\Lambda^{n}_{k} \subset \Delta^{n}$ are
standard weak equivalences, so that the inner anodyne extension
$j_{X}: X \to LX$ is a standard weak equivalence. The interval $I$ is
contractible in the standard model structure, and a map $f: Z \to W$
of quasi-categories is a weak equivalence if and only if it is an
$I$-homotopy equivalence, so that every $I$-homotopy equivalence is a
standard weak equivalence.
\end{proof}

\begin{lemma}\label{lem 30}
Suppose that $g: X \to Y$ is a categorical weak equivalence, and
that $K$ is a simplicial set. Then the map $g \times 1_{K}: X \times K
\to Y \times K$ is a categorical weak equivalence.
\end{lemma}

\begin{proof}
Suppose that $Z$ is a quasi-category. The exponential law induces a
natural bijection
\begin{equation*}
\pi_{I}(X,\mathbf{hom}(K,Z)) \cong \pi_{I}(X \times K,Z).
\end{equation*}
The function complex $\mathbf{hom}(K,Z)$ is a quasi-category by
Corollary \ref{cor 11}, so that the induced map
\begin{equation*}
g^{\ast}: \pi_{I}(Y,\mathbf{hom}(K,Z)) \to \pi_{I}(X,\mathbf{hom}(K,Z))
\end{equation*}
is a bijection.
\end{proof}

\begin{corollary}
Suppose that $i: A \to B$ is a cofibration 
and a categorical weak equivalence, and that $j: C \to D$ is a
cofibration. Then the cofibration
\begin{equation*}
(B \times C) \cup (A \times D) \subset B \times D
\end{equation*}
is a categorical weak equivalence.
\end{corollary}

\begin{example}
A functor $f:C \to D$ of small categories induces a categorical
weak equivalence $f_{\ast}: BC \to BD$ if and only if the finctor $f$
is an equivalence of categories. 

In effect, $BC$ and $BD$ are
quasicategories, so that $f_{\ast}$ is a categorical weak
equivalence if and only if it is an $I$-homotopy equivalence. This means
that there is a functor $g: D \to C$ and homotopies $C \times
\pi(\Delta^{1}) \to C$ and $D \times \pi(\Delta^{1}) \to D$ which
define the homotopies $g \cdot f \simeq 1$ and $f \cdot g \simeq 1$.

It follows that the quasi-category $BC$ is weakly equivalent to a
point if and only if $C$ is a trivial groupoid.
\end{example} 

\begin{example}\label{ex 33}
Suppose that $C$ is a small category. Then $JBC$ is the nerve
$B(Iso(C))$ of the nerve of the groupoid of isomorphisms in $C$. The
inclusion $JBC \subset BC$ is a quasi-category equivalence if and only
if $C$ is a groupoid.

Thus, in general, the map $JX \subset X$ is not a quasi-category
equivalence for quasi-categories $X$.
\end{example}

\begin{lemma}\label{lem 34}
\begin{itemize}
\item[1)]
Suppose that $p: X \to Y$ is a pseudo-fibration of quasi-categories. Then the
induced map $p: J(X) \to J(Y)$ is a Kan fibration.
\item[2)]
Suppose that $p: X \to Y$ is a trivial fibration of quasi-categories. Then the
induced map $p: J(X) \to J(Y)$ is a trivial Kan fibration.
\item[3)] 
Suppose that $f: X \to Y$ is a categorical weak equivalence of
  quasi-categories. Then the induced map $f: J(X) \to J(Y)$ is a weak
  equivalence of simplicial sets.
\end{itemize}
\end{lemma}

\begin{proof}
For statement 1), suppose given a commutative diagram
\begin{equation*}
\xymatrix{
\Lambda^{n}_{k} \ar[r] \ar[d] & J(X) \ar[r] \ar[d] & X \ar[d]^{p} \\
\Delta^{n} \ar[r] \ar@{.>}[urr]^{\theta} & J(Y) \ar[r] & Y
}
\end{equation*}
where $0 < k < n$. Then the dotted arrow $\theta$ exists since $p$ is
an inner fibration. The map $P(\Lambda^{n}_{k}) \to P(\Delta^{n})$ is
an isomorphism, so $\theta$ maps all $1$-simplices of $\Delta^{n}$ to
quasi-isomorphisms. Thus, $\theta$ factors through a map $\theta':
\Delta^{n} \to J(X)$, and so the map $p: J(X) \to J(Y)$ is an inner
fibration.

The map $p$ has the path lifting property in the sense that all
lifting problems
\begin{equation*}
\xymatrix{
\{ \epsilon \} \ar[r] \ar[d] & X \ar[d]^{p} \\
I \ar[r] \ar@{.>}[ur] & Y
}
\end{equation*}
can be solved, where $\epsilon = 0,1$. Any lift $I \to X$
has its image in $J(X)$, since $I$ is a Kan complex.  

Suppose given a
lifting problem
\begin{equation*}
\xymatrix{
\{ \epsilon \} \ar[r] \ar[d] & J(X) \ar[d]^{p} \\
\Delta^{1} \ar[r]_{\alpha} \ar@{.>}[ur] & J(Y)
}
\end{equation*}
Then there is a morphism $\alpha' : I \to J(Y)$ such that $\alpha'
\cdot \eta = \alpha$, where $\eta: \Delta^{1} \to I$ is the trivial
cofibration that we've been using, since $J(Y)$ is a Kan complex. It
follows that there is a commutative diagram
\begin{equation*}
\xymatrix{
\{ \epsilon \} \ar[rr] \ar[d] & & J(X) \ar[r] \ar[d]^{p} & X \ar[d]^{p} \\
\Delta^{1} \ar[r] & I \ar[r]_{\alpha'} \ar[ur] & J(Y) \ar[r] & Y
}
\end{equation*}
and so the map $p: J(X) \to J(Y)$ has the path lifting property.
The map $p$ is therefore a Kan fibration, by Corollary \ref{cor 18}. 

For statement 2), the trivial fibration 
$p: X \to Y$ creates quasi-isomorphisms, so that the diagram
\begin{equation*}
\xymatrix{
J(X) \ar[r] \ar[d]_{p} & X \ar[d]^{p} \\
J(Y) \ar[r] & Y
}
\end{equation*}
is a pullback. It follows that the map $p: J(X) \to J(Y)$ is a trivial
Kan fibration.

Statement 3) is a formal consequence of statement 2), by the usual
factorization trick: the map $f$ is a composite $f = q \cdot j$ where
$q$ is a pseudo-fibration and a categorical weak equivalence (a
trivial fibration), and $j$ is a section of a trivial fibration.
\end{proof}

\begin{remark}
Suppose that $f: X \to Y$ is a map of quasi-categories such that
$f_{\ast}: J(X) \to J(Y)$ is a weak equivalence of simplicial sets. It
does not follow that $f$ is a categorical weak equivalence.

For example, suppose that $C$ is a small category and recall that the
core $J(BC)$ of the nerve of $C$ is the nerve $B(\Iso(C))$ of the
groupoid of isomorphisms of $C$. The map $B(\Iso(C)) \to BC$ induces
an isomorphism of cores, but is not a categorical weak equivalence in
general, since $C$ may not be a groupoid. See also Example \ref{ex 33}.
\end{remark}

The fibrant model construction $j: X \to LX$ for the quasi-category
model structure is defined by a countable sequence of cofibrations
\begin{equation*}
X=X_{0} \to X_{1} \to X_{2} \to \dots
\end{equation*}
with $LX = \varinjlim_{i}\ X_{i}$. In all cases, $X_{i+1}$ is
constructed from $X_{i}$ by forming the pushout
\begin{equation*}
\xymatrix{
\bigsqcup_{\Lambda^{n}_{k} \to X_{i}}\ \Lambda^{n}_{k} \ar[r] \ar[d] & X_{i} \ar[d] \\
\bigsqcup_{\Lambda^{n}_{k} \to X_{i}}\ \Delta^{n} \ar[r] & X_{i+1}
}
\end{equation*}
where the disjoint union is indexed over all maps $\Lambda^{n}_{k} \to
X_{i}$ of inner horns to $X_{i}$. 

Suppose that $\alpha$ is a regular cardinal. It is a consequence of
the construction that if $X$ is $\alpha$-bounded, then $LX$ is
$\alpha$-bounded. Each of the functors $X \mapsto X_{i}$ preserves
monomorphisms and filtered colimits, and it follows that the functor
$X \mapsto LX$ has these same properties. If $A$ and $B$ are
subobjects of a simplicial set $X$, then there is an isomorphism
\begin{equation*}
L(A \cap B) \xrightarrow{\cong} LA \cap LB,
\end{equation*}
as subobjects of $LX$.

\begin{lemma}
Suppose that $i: X \to Y$ is a cofibration and a categorical weak
equivalence. Suppose that $A \subset Y$ is an $\alpha$-bounded
subobject of $Y$. Then there is an $\alpha$-bounded subobject $B$ of
$Y$ with $A \subset B$, such that the map $B \cap X \to B$ is a
categorical weak equivalence.
\end{lemma}

This result is a consequence of the method of proof of Lemma 4.9 of
\cite{J40}. 

\begin{proof}
The map $i_{\ast}: LX \to LY$ is a filtered colimit of the maps $L(B \cap X) \to LB$, indexed over the $\alpha$-bounded subobjects $B$ of $Y$. All diagrams
\begin{equation*}
\xymatrix{
L(B \cap X) \ar[r] \ar[d] & LX \ar[d]^{i_{\ast}} \\
LB \ar[r] & LY
}
\end{equation*}
are pullbacks. Every map $f: Z \to W$ between quasi-categories has a
functorial factorization $f = p \cdot j$, where $p$ is a
pseudo-fibration and $j$ is a section of a trivial
fibration. For the map $LX \to LY$ this factorization
has the form
\begin{equation}\label{eq 5}
\xymatrix{
LX \ar[r]^{j} \ar[dr] & Z \ar[d]^{p} \\
& LY
}
\end{equation}
For each $\alpha$-bounded subobject $B \subset Y$, the
factorization of the map $L(B \cap X) \to LB$ can be written
\begin{equation}\label{eq 6}
\xymatrix{
L(B \cap X) \ar[r]^{j_{B}} \ar[dr] & Z_{B} \ar[d]^{p_{B}} \\
& LB
}
\end{equation}
The diagram (\ref{eq 5}) is a filtered colimit of diagrams 
(\ref{eq 6}). It follows that all lifting problems 
\begin{equation*}
\xymatrix{
\partial\Delta^{n} \ar[r] \ar[d] & Z_{A} \ar[d] \\
\Delta^{n} \ar[r] \ar@{.>}[ur] & LA
}
\end{equation*}
have solutions over some $LB_{1}$, where $B_{1}$ is an
$\alpha$-bounded subcomplex of $Y$ such that $A \subset
B_{1}$. Continue inductively, to produce a chain of $\alpha$-bounded
subobjects
\begin{equation*}
A \subset B_{1} \subset B_{2} \subset \dots
\end{equation*}
such that all lifting problems
\begin{equation*}
\xymatrix{
\partial\Delta^{n} \ar[r] \ar[d] & Z_{B_{i}} \ar[d] \\
\Delta^{n} \ar[r] \ar@{.>}[ur] & LB_{i}
}
\end{equation*}
have solutions over $LB_{i+1}$. 

Set $B= \varinjlim_{i}\ B_{i}$. Then the map $p_{B}: Z_{B} \to LB$ is
a trivial fibration, so the map $L(B \cap X) \to LB$ is a
quasi-category weak equivalence.
\end{proof}

\begin{lemma}\label{lem 37}
Suppose that $f: X \to Y$ is a categorical weak equivalence of
quasi-categories and that $A$ is a simplicial set. Then the map
\begin{equation*}
\mathbf{hom}(A,X) \to \mathbf{hom}(A,Y)
\end{equation*}
is a categorical weak equivalence of quasi-categories.
\end{lemma}

\begin{proof}
The object $\mathbf{hom}(A,X)$ is a quasi-category if $X$ is a
quasi-category, by Corollary \ref{cor 11}.

The map $f$ has a factorization $f = q \cdot j$, where $q$ is a
trivial fibration and $j$ is a section of a trivial fibration. The
functor $X \mapsto \mathbf{hom}(A,X)$ preserves trivial fibrations,
and therefore preserves categorical weak equivalences between
quasi-categories.
\end{proof}

\begin{proposition}\label{prop 38}
A map $f: X \to Y$ between quasi-categories is a categorical weak
equivalence if and only if it induces equivalences of groupoids
\begin{equation}\label{eq 7}
\begin{aligned}
&f_{\ast}: \pi J(\mathbf{hom}(\partial\Delta^{n},X)) \to 
\pi J(\mathbf{hom}(\partial\Delta^{n},Y))\ \text{and}\\
&f_{\ast}: \pi J(\mathbf{hom}(\Delta^{n},X)) \to 
\pi J(\mathbf{hom}(\Delta^{n},Y)) \\
\end{aligned}
\end{equation}
for $n \geq 1$.
\end{proposition}

\begin{proof}
Suppose that $f: X \to Y$ is a quasi-weak equivalence of
quasi-categories. Then all induced maps
\begin{equation*}
\mathbf{hom}(A,X) \to \mathbf{hom}(A,Y)
\end{equation*}
are quasi-weak equivalences of quasi-categories by Lemma 
\ref{lem 37}. The induced maps
\begin{equation*}
J\mathbf{hom}(A,X) \to J\mathbf{hom}(A,Y)
\end{equation*}
are weak equivalences of Kan complexes by Lemma \ref{lem 34}, and
therefore induce equivalences of fundamental groupoids
\begin{equation*}
\pi J\mathbf{hom}(A,X) \to \pi J\mathbf{hom}(A,Y)
\end{equation*}
It follows that the maps of (\ref{eq 7}) are equivalences of groupoids.

For the converse, suppose that all morphisms (\ref{eq 7}) are weak
equivalences of groupoids.

It suffices to assume that $f$ is a pseudo-fibration, by the
usual fibration replacement trick for maps between Kan complexes: if
$f = p \cdot i$ where $i$ is a section of a trivial fibration and $p$
is a pseudo-fibration, then $f$ satisfies the conditions of
the Lemma if and only if $p$ does so.

Suppose, therefore, that $f$ is a pseudo-fibration. We want to
solve the lifting problem
\begin{equation*}
\xymatrix{
\partial\Delta^{n} \ar[r]^{\alpha} \ar[d]_{i} & X \ar[d]^{f} \\ 
\Delta^{n} \ar[r]_{\beta} \ar@{.>}[ur] & Y
}
\end{equation*}

The map
\begin{equation*}
f_{\ast}: \mathbf{hom}(\Delta^{n},X) \to \mathbf{hom}(\Delta^{n},Y)
\end{equation*}
is a pseudo-fibration by Lemma \ref{lem 30}. The map
\begin{equation*}
f_{\ast}: \pi J(\mathbf{hom}(\Delta^{n},X)) \to \pi J(\mathbf{hom}(\Delta^{n},Y))
\end{equation*}
is an equivalence of groupoids by assumption, so that there is a
quasi-isomor\-phism $\Delta^{1} \to \mathbf{hom}(\Delta^{n},Y)$ from the
vertex $\beta$ to $p(\gamma)$ for some $\gamma: \Delta^{n}
\to X$.

All pseudo-fibrations $p: Z \to W$ satisfy the path lifting property
\begin{equation*}
\xymatrix{
\{ \epsilon \} \ar[r] \ar[d] & Z \ar[d]^{p} \\
I \ar[r] \ar@{.>}[ur] & W
}
\end{equation*}
It follows that all pseudo-fibrations $p$ between quasi-categories satisfy a
lifting property
\begin{equation*}
\xymatrix{
\{ \epsilon \} \ar[r] \ar[d] & Z \ar[d]^{p} \\
\Delta^{1} \ar[r]_{\gamma} \ar@{.>}[ur]^{\theta} & W
}
\end{equation*}
for quasi-isomorphisms $\gamma$, since a quasi-isomorphism $\gamma$
extends to a morphism
\begin{equation*}
I \to J(W) \subset W.
\end{equation*}
We can further assume that the lifting $\theta$ is a quasi-isomorphism
of $Z$. 

It follows that there is a quasi-isomorphism $\Delta^{1} \to
\mathbf{hom}(\Delta^{n},X)$ from a pre-image $\beta'$ of $\beta$ to
$\gamma$, and in particular there is a lifting 
\begin{equation*}
\xymatrix{
& X \ar[d]^{f} \\
\Delta^{n} \ar[r]_{\beta} \ar[ur]^{\beta'} & Y 
}
\end{equation*}

The map
\begin{equation*}
f_{\ast}: J(\mathbf{hom}(\partial\Delta^{n},X)) \to 
J(\mathbf{hom}(\partial\Delta^{n},Y))
\end{equation*}
is a Kan fibration by Lemma \ref{lem 34}.
The vertices $\beta' \cdot i$ and $\alpha$ of the Kan complex
$J(\mathbf{hom}(\partial\Delta^{n},X))$ have the same image, namely
$\beta \cdot i$ under this fibration $f_{\ast}$, and are therefore in
the fibre $F_{\beta\cdot i}$ of $f_{\ast}$ over $\beta\cdot i$. The
fibration $f_{\ast}$ induces an equivalence of fundamental groupoids
by assumption, so that the fibre $F_{\beta\cdot i}$ is a connected Kan
complex. It follows that there is a path $h: \Delta^{1} \to F_{\beta
  \cdot i}$ from $\alpha$ to $\beta' \cdot i$, and so there is a map
$h'$ making the diagram
\begin{equation*}
\xymatrix{
\Delta^{1} \ar[r]^{h} \ar[d]_{\eta} & F_{\beta \cdot i} \\
I \ar[ur]_{h'}
}
\end{equation*}
since the map $\eta$ is a trivial cofibration in the standard model
structure for simplicial sets. Write $H$ for the adjoint of the composite
\begin{equation*}
I \xrightarrow{h'} F_{\beta \cdot i} \to J(\mathbf{hom}(\partial\Delta^{n},X).
\end{equation*}
Then the lifting problem
\begin{equation*}
\xymatrix{
(\Delta^{n} \times \{ 1 \}) \cup (\partial\Delta^{n} \times I) 
\ar[r]^-{(\beta',H)} \ar[d]  
& X \ar[d]^{f} \\
\Delta^{n} \times I \ar[r] \ar@{.>}[ur]^{\theta} & Y
}
\end{equation*}
can be solved since $f$ is a pseudo-fibration, by Lemma 
\ref{lem 22}. It follows that there is a commutative diagram
\begin{equation*}
\xymatrix{
\partial\Delta^{n} \ar[r]^{\alpha} \ar[d]_{i} & X \ar[d]^{f} \\ 
\Delta^{n} \ar[r]_{\beta} \ar[ur]^{\theta'} & Y
}
\end{equation*}
where $\theta'$ is the composite
\begin{equation*}
\Delta^{n} \times \{ 0 \} \subset \Delta^{n} \times I \xrightarrow{\theta} X.
\end{equation*}
\end{proof}

\begin{corollary}\label{cor 39}
A map $f: X \to Y$ between Kan complexes is a standard weak
equivalence of simplicial sets if and only if it induces equivalences
of groupoids
\begin{equation}\label{eq 8}
f_{\ast}: \pi(\mathbf{hom}(\partial\Delta^{n},X)) \to 
\pi(\mathbf{hom}(\partial\Delta^{n},Y))
\end{equation}
for $n \geq 1$.
\end{corollary}

\begin{proof}
Suppose that all maps (\ref{eq 8}) are weak equivalences of
groupoids. Then the morphism $\pi(X) \to
\pi(Y)$ is an equivalence of fundamental groupoids, since it is a retract of the equivalence
\begin{equation*}
\pi\mathbf{hom}(\partial\Delta^{1},X) \to \pi\mathbf{hom}(\partial\Delta^{1},Y).
\end{equation*}
Any vertex $\ast \to
\Delta^{n}$ induces a natural weak equivalence
\begin{equation*}
\mathbf{hom}(\Delta^{n},X) \xrightarrow{\simeq} X
\end{equation*}
of Kan complexes, while there is an identification
\begin{equation*}
J(\mathbf{hom}(\Delta^{n},X)) = \mathbf{hom}(\Delta^{n},X).
\end{equation*}
It follows that all induced groupoid morphisms
\begin{equation*}
\pi J(\mathbf{hom}(\Delta^{n},X)) \to 
\pi J(\mathbf{hom}(\Delta^{n},Y))
\end{equation*}
are equivalences. 

It follows from Proposition \ref{prop 38} that the map $f: X \to Y$ is a
categorical weak equivalence. The Kan complexes $X$ and $Y$ are
quasi-categories, so Lemma \ref{lem 29} implies that $f$ is a
standard weak equivalence.
\end{proof}

Corollary \ref{cor 39} can also be proved directly, by using traditional
methods of simplicial homotopy theory.
\medskip

It follows from Lemma 4.13 of \cite{J40} that a map $p: X \to Y$
between quasi-categories is a pseudo-fibration if and only if it has
the right lifting property with respect to the maps (\ref{eq 2}) and
(\ref{eq 3}). 

The maps (\ref{eq 2}) induce isomorphisms of path
categories, and the map (\ref{eq 3}) induces an isomorphism of path
categories
\begin{equation*}
P((\partial\Delta^{n} \times \square^{m}) \cup (\Delta^{n} \times \sqcap^{m}_{(k,\epsilon)})) \xrightarrow{\cong} P(\Delta^{n} \times \square^{m})
\end{equation*}
if $m \geq 1$.  It follows that a functor $p: C \to D$ between small
categories induces a pseudo-fibration if and only if it has the right
lifting property with respect to all functors $\mathbf{n} \times \{
\epsilon \} \to \mathbf{n} \times \pi(\Delta^{1})$, where $\epsilon = 0,1$. 

It is then an exercise to show that the functor $p: C \to D$ defines a
pseudo-fibration $BC \to BD$ if and only if it has the {\it
  isomorphism lifting property} in the sense that all lifting problems
\begin{equation*}
\xymatrix{
\{ \epsilon \} \ar[r] \ar[d] & C \ar[d]^{p} \\
\mathbf{1} \ar[r]_{\alpha} \ar@{.>}[ur] & D
}
\end{equation*}
have solutions, where $\epsilon = 0,1$, and the morphism defined by the
functor $\alpha$ is an isomorphism of $D$.

If $p: C \to D$ has the isomorphism lifting property,
then the induced functors $p: C^{\mathbf{n}} \to D^{\mathbf{n}}$ have
the isomorphism lifting property, for all $n \geq 1$.

\begin{example}\label{ex 40}
Suppose that the functor $p: C \to D$ has the isomorphism lifting
property, and suppose that the induced map $p_{\ast}: BC \to BD$ of
quasi-categories satisfies the criteria for a categorical weak
equivalence that are given by Proposition \ref{prop 38}. These
criteria mean, precisely, 
that all induced functors 
\begin{equation}\label{eq 9}
\Iso(C^{\mathbf{n}}) \to
\Iso(D^{\mathbf{n}})
\end{equation}
 and
\begin{equation}\label{eq 10}
\Iso(C^{P(\partial\Delta^{n})}) \to \Iso(D^{P(\partial\Delta^{n})})
\end{equation} 
are equivalences of groupoids for $n \geq 0$. These functors are then trivial
fibrations of groupoids in the traditional sense, by Lemma 
\ref{lem 34}.

To solve the lifting problems 
\begin{equation*}
\xymatrix{
\partial\Delta^{n} \ar[r] \ar[d] & BC \ar[d]^{p_{\ast}} \\
\Delta^{n} \ar[r] \ar@{.>}[ur] & BD
}
\end{equation*}
it suffices to assume that $n \leq 2$, since the induced functors
$P(\partial\Delta^{n}) \to P(\Delta^{n})$ are isomorphisms for $n \geq 3$.

Every object $x$ of $D$ is isomorphic to the image $p(y)$ of some
object $y$ of $C$ since the map $\Iso(C) \to \Iso(D)$ is an equivalence of groupoids. We make a specific choice of isomorphism
$\alpha: x \xrightarrow{\cong} p(y)$ in $D$. The lifting problem
\begin{equation*}
\xymatrix{
\{ 1 \} \ar[r]^{y} \ar[d] & C \ar[d]^{p} \\
\mathbf{1} \ar[r]_{\alpha} \ar@{.>}[ur] & D
}
\end{equation*}
has a solution, so there is an object $z$ of $C$ such that $p(z) = x$.

Suppose given a lifting problem
\begin{equation*}
\xymatrix{
\partial\Delta^{2} \ar[r]^{\alpha} \ar[d] & BC \ar[d]^{p_{\ast}} \\
\Delta^{2} \ar[r]_{\beta} \ar@{.>}[ur] & BD
}
\end{equation*}
The map $\Iso(C^{\mathbf{2}}) \to
\Iso(D^{\mathbf{2}})$ is an equivalence of groupoids, so there is an
isomorphism $h: \mathbf{1} \to \Iso(D^{\mathbf{2}})$ from $\beta$ to a
functor $p(\gamma)$, for some $\gamma: \mathbf{2} \to C$. The
  isomorphism $h$ lifts along the pseudo-fibration $C^{\mathbf{2}} \to
  D^{\mathbf{2}}$ to an isomorphism $H: \omega \xrightarrow{\cong}
  \gamma$ in $C^{\mathbf{2}}$. The restriction
  $\omega\vert_{\partial\Delta^{2}}$ and $\alpha$ have the same image,
  namely $\beta\vert_{\partial\Delta^{2}}$ under $p$, and the map $p:
  \Iso(C^{P(\partial\Delta^{2})}) \to \Iso(D^{P(\partial\Delta^{2})})$
  is a trivial fibration of groupoids. It follows that there is an
  isomorphism $\omega\vert_{\partial\Delta^{2}} \to \alpha$ of
  $\Iso(C^{P(\partial\Delta^{2})})$ which maps to the identity of
  $\beta\vert_{\partial\Delta^{2}}$ under $p$. It follows that there
  is a (unique) functor $\zeta: \mathbf{2} \to C$ which maps to
  $\beta$ under $p$ and restricts to $\alpha$ on $\partial\Delta^{2}$.

The solution of the lifting problem
\begin{equation*}
\xymatrix{
\partial\Delta^{1} \ar[r]^{\alpha} \ar[d] & BC \ar[d]^{p_{\ast}} \\
\Delta^{1} \ar[r]_{\beta} \ar@{.>}[ur] & BD
}
\end{equation*}
is very similar.

The moral is that a functor $f: C
\to D$ is an equivalence of categories if and only if the functors
(\ref{eq 9}) and (\ref{eq 10}) are equivalences of groupoids for $0
\leq n \leq 2$.
\end{example}

\section{Products of simplices}

\begin{lemma}
Suppose that $\sigma: \Delta^{n} \to \Delta^{r} \times \Delta^{s}$ is
a non-degenerate simplex. Then $n \leq r+s$.
\end{lemma}

\begin{proof}
Suppose that $\sigma$ is defined by the path
\begin{equation*}
(i_{0},j_{0}) \to (i_{1},j_{1}) \to \dots \to (i_{n},j_{n})
\end{equation*}
in the product poset $\mathbf{r} \times \mathbf{s}$. The simplex
$\sigma$ is non-degenerate, so there are no repeats in the
string. Thus, $i_{0} < i_{1}$ or $j_{0} < j_{1}$. 

If $i_{0} < i_{1}$
then the string
\begin{equation*}
d_{0}(\sigma):\ (i_{1},j_{1}) \to \dots \to (i_{n},j_{n})
\end{equation*} 
lies in a subobject of $\mathbf{r} \times \mathbf{s}$ isomorphic to a
poset $\mathbf{r}' \times \mathbf{s}$, where $r' < r$, and inductively
$d_{0}(\sigma)$ has length $L(d_{0}(\sigma))$ bounded above by
$r'+s$. Thus, $\sigma$ has length 
\begin{equation*}
n = 1 + L(d_{0}(\sigma)) \leq 1 + (r'+s) \leq r+s.
\end{equation*}
The same outcome obtains (with a similar argument) if $j_{0} < j_{1}$.
\end{proof}

\begin{corollary}
The non-degenerate simplices $\sigma: \Delta^{n} \to \Delta^{r} \times
\Delta^{s}$ of maximal dimension have dimension $n=r+s$.
\end{corollary}

Suppose that 
\begin{equation}\label{eq 11}
\sigma: (0,0) = (i_{0},j_{0}) \to \dots \to (i_{n},j_{n}) = (r,s) 
\end{equation}
 is a non-degenerate simplex of maximal dimension $n=r+s$.
Then 
$i_{k+1} \leq i_{k}+1$ for $k < r$, for otherwise, there is a
non-degenerate path
\begin{equation*}
(i_{k},j_{k}) \to (i_{k}+1,j_{k}) \to (i_{k+1},j_{k+1})
\end{equation*}
having the path $(i_{k},j_{k}) \to (i_{k+1},j_{k+1})$ as a face, and
$\sigma$ does not have maximal length. Similarly (or dually), $j_{k+1}
\leq j_{k}+1$ for $k < r$. Observe also that 
\begin{itemize}
\item[a)]
if $i_{k+1}=i_{k}$ then
$j_{k} < j_{k+1}$ so that $j_{k}+1 = j_{k+1}$, and
\item[b)] (dually) if $j_{k+1}
=j_{k}$ then $i_{k+1}=i_{k}+1$.
\end{itemize}

Suppose that
\begin{equation*}
\gamma: (0,0) = (i'_{0},j'_{0}) \to \dots \to (i'_{n},j'_{n}) = (r,s)
\end{equation*}
is another simplex of maximal dimension $n=r+s$ in
$\Delta^{r} \times \Delta^{s}$. Say that $\sigma \leq \gamma$ if
$i_{k} \leq i'_{k}$ for $0 \leq k \leq n$. 

\begin{lemma}
For the ordering $\sigma \leq \gamma$ on the set of non-degenerate
simplices of $\Delta^{r} \times \Delta^{s}$ of maximal dimension, the
simplex
\begin{equation*}
(0,0) \to (0,1) \to \dots \to (0,s) \to (1,s) \to \dots \to (r,s)
\end{equation*}
is minimal, and the simplex
\begin{equation*}
(0,0) \to (1,0) \to \dots \to (r,0) \to (r,1) \to \dots \to (r,s)
\end{equation*}
is maximal. 
\end{lemma}

\begin{proof}
If $\sigma$ is a path as (\ref{eq 11}), then there are relations
\begin{equation*}
0=i_{0} \leq i_{1} \leq \dots \leq i_{r}
\end{equation*}
and $i_{k+1} \leq i_{k}+1$.  It follows that $0 \leq i_{j} \leq j$ for $0
\leq j \leq r$. Also, $i_{j} \leq r$ for all $j$, and hence for
all $j > r$. It follows that the indicated simplex is maximal.  

The relation
\begin{equation*}
i_{k} + j_{k} = k,
\end{equation*}
holds for all non-degenerate simplices of maximal dimension --- this
is a consequence of the observations in statements a) and b)
above. Thus, for $0 \leq k \leq r$, $i_{s+k} + j_{s+k} = s+k$ and
$j_{s+k} \leq s$ together force $i_{s+k} = (s-j_{s+k}) + k \geq k$.
The assertion that the indicated simplex is minimal follows.
\end{proof}

Suppose that $\sigma$ is a non-degenerate $(r+s)$-simplex such that
the path defining $\sigma$ contains a segment
\begin{equation*}
\xymatrix{
& (i_{k}+1,j_{k}+1) \\
(i_{k},j_{k}) \ar[r] & 
(i_{k}+1,j_{k}) \ar[u] 
}
\end{equation*}
Then the path $\sigma'$ that is obtained from $\sigma$ by replacing
the segment above by the path
\begin{equation}\label{eq 12}
\xymatrix{
(i_{k},j_{k}+1) \ar[r] & (i_{k}+1,j_{k}+1) \\
(i_{k},j_{k}) \ar[u]
}
\end{equation}
satisfies $\sigma' \leq \sigma$. 

Observe that $d_{k+1}(\sigma) = d_{k+1}(\sigma')$, and that $\sigma$
and $\sigma'$ are the only two non-degenerate $(r+s)$-simplices for
which $d_{k+1}(\sigma)$ could be a face. In particular, if $\tau <
\sigma'$ then $d_{k+1}(\sigma)$ is not a face of $\tau$.

\begin{lemma}\label{lem 44}
Suppose that $\sigma$ is a non-degenerate $(r+s)$-simplex of
$\Delta^{r} \times \Delta^{s}$.
\begin{itemize}
\item[1)]
If $\sigma$ is not maximal, then it
contains a segment
\begin{equation*}
\xymatrix{
(i_{k},j_{k}+1) \ar[r] & (i_{k}+1,j_{k}+1) \\
(i_{k},j_{k}) \ar[u]
}
\end{equation*}
\item[2)]
If $\sigma$ is not minimal, then it contains a segment
\begin{equation*}
\xymatrix{
& (i_{r}+1,j_{r}+1) \\
(i_{r},j_{r}) \ar[r] & (i_{r}+1,j_{r}) \ar[u]
}
\end{equation*}
\end{itemize}
\end{lemma}

\begin{proof}
We prove statement 1).  The proof of
statement 2) is similar.

We argue by induction on $r+s \geq 2$.

If $r=s=1$ there are two non-degenerate $2$-simplices in
$\Delta^{1} \times \Delta^{1}$ and the one that is not maximal has
the form (\ref{eq 12}). 

Suppose that $i_{k}$ is minimal such that $(i_{k},j_{k}) =
(i_{k},s)$. If $i_{k} < r$, then $\sigma$ has a segment
\begin{equation*}
\xymatrix{
(i_{k},r) \ar[r] & (i_{k}+1,r) \\
(i_{k},r-1) \ar[u]
}
\end{equation*}

Suppose that $i_{k}=r$. Choose the minimal $i_{p}$ such that
$(i_{p},j_{p}) = (r,j_{p})$. Then $i_{p} > 0$ since $\sigma$ is not
maximal, and the segment of $\sigma$ ending at $(r,j_{p}))$ has the form
\begin{equation*}
(0,0) \to \dots \to (r-1,j_{p}) \to (r,j_{p}).
\end{equation*}
This segment defines a maximal non-degenerate simplex of $\Delta^{r}
\times \Delta^{j_{p}}$, which is susceptible to the argument of the
first paragraph. This simplex therefore has a segment of the required
form, as does the simplex $\sigma$.
\end{proof}

An {\it interior simplex} of $\Delta^{r} \times \Delta^{s}$ is a poset morphism
\begin{equation*}
\mathbf{m} \xrightarrow{\theta} \mathbf{r} \times \mathbf{s}
\end{equation*}
such that the composites
\begin{equation*}
\begin{aligned}
&\mathbf{m} \xrightarrow{\theta} \mathbf{r} \times \mathbf{s} \to \mathbf{r} \\
&\mathbf{m} \xrightarrow{\theta} \mathbf{r} \times \mathbf{s} \to \mathbf{s}
\end{aligned}
\end{equation*}
with the respective projections are surjective. Such a simplex
$\theta$ cannot lie in the boundary subcomplex
\begin{equation*}
(\partial\Delta^{r} \times \Delta^{s}) \cup 
(\Delta^{r} \times \partial\Delta^{s}),
\end{equation*}
for then one of the two composites above would fail to be surjective.

\begin{theorem}\label{th 45}
Suppose that $0 < k < n$. Then the inclusion
\begin{equation*}
(\Lambda^{n}_{k} \times \Delta^{m}) \cup (\Delta^{n} \times \partial\Delta^{m})
\subset \Delta^{n} \times \Delta^{m}
\end{equation*}
is inner anodyne.
\end{theorem}

\begin{proof}
Suppose that $T$ is a set of non-degenerate $(m+n)$-simplices of
$\Delta^{m} \times \Delta^{n}$ that is order closed in the sense that
if $\gamma \in T$ and $\tau \leq \gamma$ then $\tau \in T$.

Given $T$ pick a smallest $\sigma$ such that $\sigma \notin T$. Then
the collection of all simplices $\tau$ such that $\tau < \sigma$ is
contained in $T$ and the set $T' = T \cup \{ \sigma \}$ is order
closed. The empty set $\emptyset$ of non-degenerate $(m+n)$-simplices
is the minimal order-closed set, and the full set of non-degenerate
$(m+n)$-simplices is the maximal order-closed set. 

Write 
\begin{equation*}
(\Delta^{n} \times \Delta^{m})^{(T)} 
\end{equation*}
for the subcomplex of $\Delta^{m} \times \Delta^{n}$ that is generated by the subcomplex
\begin{equation*}
(\Delta^{n} \times \Delta^{m})^{(\emptyset)} = (\Lambda^{n}_{k} \times \Delta^{m}) \cup (\Delta^{n} \times \partial\Delta^{m})
\end{equation*}
and the simplices in $T$.

The idea of the proof is to show that all inclusions
\begin{equation*}
(\Delta^{n} \times \Delta^{m})^{(T)}  \subset 
(\Delta^{n} \times \Delta^{m})^{(T')} 
\end{equation*}
are inner anodyne for order closed sets $T$ and $T' = T \cup \{ \sigma
\}$ as above.

Every non-degenerate $(m+n)$-simplex $\sigma$ has $d_{0}(\sigma)$ and
$d_{m+n}(\sigma)$ in the subcomplex $(\Delta^{n} \times
\Delta^{m})^{(\emptyset)}$: 
\begin{itemize}
\item[1)]
If the first member in the path
\begin{equation*}
\sigma:\ (0,0) = (i_{0},j_{0}) \to \dots \to (i_{m+n},j_{m+n}) = (n,m)
\end{equation*}
is the morphism $(0,0) \to (1,0)$ then $d_{0}(\sigma)$ is in the image
of the poset morphism $d^{0} \times 1: (\mathbf{n-1}) \times \mathbf{m}
\to \mathbf{n} \times \mathbf{m}$, which is in $\Lambda^{n}_{k} \times
\Delta^{m}$ since $k \ne 0$.  If the first member of the path $\sigma$
is the morphism $(0,0) \to (0,1)$ then $d_{0}(\sigma)$ is in the image
of the morphism $1 \times d^{0}: \mathbf{n} \times (\mathbf{m-1}) \to \mathbf{n} \times \mathbf{m}$,
which is in $\Delta^{n} \times \partial\Delta^{m}$.
\item[2)]
If the last morphism in the path $\sigma$ is the morphism $(n-1,m) \to
(n,m)$ then $d_{n+m}(\sigma)$ is in the image of the poset morphism
$d^{n} \times 1: (\mathbf{n-1}) \times \mathbf{m} \to \mathbf{n} \times
\mathbf{m}$, which is in $\Lambda^{n}_{k} \times \Delta^{m}$ since $k
\ne n$. If the last morphism of $\sigma$ is $(n,m-1) \to (n,m)$ then
$d_{n+m}(\sigma)$ is in the image of the morphism $1 \times d^{m}:
\mathbf{n} \times (\mathbf{m-1}) \to \mathbf{n} \times \mathbf{m}$, which
is in $\Delta^{n} \times \partial\Delta^{m}$.
\end{itemize}

There is a pushout diagram
\begin{equation*}
\xymatrix{
K \ar[r] \ar[d] & (\Delta^{n} \times \Delta^{m})^{(T)} \ar[d]^{i} \\
\Delta^{m+n} \ar[r]_-{\sigma} & (\Delta^{n} \times \Delta^{m})^{(T')}
}
\end{equation*}
where 
\begin{equation*}
K = \Delta^{m+n} \cap (\Delta^{n} \times \Delta^{m})^{(T)}
\end{equation*}
in $\Delta^{n} \times \Delta^{m}$. 

I claim that  the maximal non-degenerate
simplices of $K$ have dimension $m+n-1$ so that $K= \langle S \rangle$
some set $S$ of non-degenerate $(m+n-1)$-simplices of $\Delta^{m+n}$. 

To see this, let
$\gamma$ be such a maximal non-degenerate simplex of $K$, and write
\begin{equation*}
(\gamma(0),\gamma'(0)) \to \dots \to (\gamma(r),\gamma'(r))
\end{equation*}
for the string defining $\gamma$. 

If $(\gamma(0),\gamma'(0)) \ne
(0,0)$ then $\gamma$ is contained in the image of one of the morphisms
$d^{0} \times 1$ or $1 \times d^{0}$, and the maximum length of
non-degenerate simplices in these images in $m+n-1$. It follows that
$\gamma = d_{0}(\sigma)$. 
Thus, we can assume (for otherwise $\gamma$ is a face of a simplex of dimension $m+n-1$ of $K$) that $(\gamma(0),\gamma'(0)) = (0,0)$. We can
similarly assume that $(\gamma(r),\gamma'(r))= (m,n)$.

Suppose that the morphism
\begin{equation*}
(\gamma(i),\gamma'(i)) \to (\gamma(i+1),\gamma'(i+1))
\end{equation*}
is one of the morphisms in the string defining $\gamma$. This morphism
is the composite of the segment of morphisms
\begin{equation*}
\tau_{i}:\ (i_{1},j_{1}) \to \dots \to (i_{k},j_{k})
\end{equation*}
appearing  in  the  string   defining  $\sigma$.  This  segment  is  a
non-degenerate  simplex  of $\Delta^{r}  \times  \Delta^{s}$ for  some
$r,s$. 
If this string $\tau_{i}$ is not minimal among all such simplices then it
contains a substring of the form
\begin{equation*}
\xymatrix{
& (i_{v}+1,j_{v}+1) \\
(i_{v},j_{v}) \ar[r] & (i_{v}+1,j_{v}) \ar[u] 
}
\end{equation*}
by Lemma \ref{lem 44}.
It follows that $\gamma$ is a face of $d_{v+1}(\sigma)$,
which is a face of $\sigma$ as well as a face of some (unique)
$\sigma' < \sigma$, so $\gamma = d_{v+1}(\sigma)$ is an
$(n+m-1)$-simplex of $K$.

We can therefore assume that all of the strings $\tau_{i}$ are
minimal.  It follows that the string $\sigma$ is minimal among strings
of length $n+m$ passing through all of the points
$(\gamma(i),\gamma'(i))$. It follows that
$\gamma$ is a face of some simplex in $(\Delta^{n} \times
\Delta^{m})^{(\emptyset)}$, and hence that $\gamma$ is not interior.

If the composite
\begin{equation*}
\mathbf{r} \xrightarrow{\gamma} \mathbf{n} \times \mathbf{m} \to \mathbf{n}
\end{equation*}
is not surjective, then there is an $i$ for which $\gamma(i) + 1 <
\gamma(i+1)$. This means that the corresponding segment of $\sigma$ has the form
\begin{equation*}
\xymatrix@C=10pt{
(\gamma(i),\gamma'(i)) \ar[r] & \dots \ar[r] &  (\gamma(i),\gamma'(i)+s) \ar[d]  \\
&& (\gamma(i)+1,\gamma'(i)+s) \ar[r] & \dots \ar[r] 
& (\gamma(i)+r,\gamma'(i)+s)
}
\end{equation*}
where $r>1$. But then $d_{v}(\sigma)$ is in the image of some $d^{j}
\times 1: (\mathbf{n-1}) \times \mathbf{m} \to \mathbf{n} \times
\mathbf{m}$ (for some $v$ determined by the point
$(\gamma(i)+1,\gamma'(i)+s)$ in the string above), so that $\gamma =
d_{v}(\sigma)$ has dimension $m+n-1$. A similar argument shows that
$\gamma$ has dimension $m+n-1$ if the other composite is not
surjective.
\medskip

With all of that in hand, suppose that $\sigma$ is not
maximal. Then $\sigma$ contains a segment
\begin{equation*}
\xymatrix{
(i_{r},j_{r}+1) \ar[r] & (i_{r}+1,j_{r}+1) \\
(i_{r},j_{r}) \ar[u] & \\
}
\end{equation*}
and the face $d_{r+1}(\sigma)$ is not in any smaller non-degenerate
$(n+m)$-simplex and is interior. There is a pushout diagram
\begin{equation}\label{eq 13}
\xymatrix{
\langle S \rangle \ar[r] \ar[d] 
& (\Delta^{n} \times \Delta^{m})^{(T)} \ar[d]^{i} \\
\Delta^{n+m} \ar[r]_-{\sigma} & (\Delta^{n} \times \Delta^{m})^{(T')}
}
\end{equation}
where $S$ is a set of non-degenerate $(n+m-1)$-simplices of
$\Delta^{n+m}$ that includes the simplices $d^{0}$ and $d^{n+m}$ but
is missing $d^{r+1}$. It follows from Lemma \ref{lem 8} that the morphism
$i$ is inner anodyne.

If $\sigma$ is maximal, there is still a pushout diagram of the form
(\ref{eq 13}), and the faces $d^{0}$ and $d^{m+n}$ are still in the set
$S$, but $d^{k} = d_{k}(\sigma)$ is not a member of $S$. In effect,
$d_{k}(\sigma) \mapsto d^{k}$ under the projection $\Delta^{n} \times
\Delta^{m} \to \Delta^{m}$ so it is not in $\Lambda^{n}_{k} \times
\Delta^{n}$, and the composite
\begin{equation*}
\mathbf{n+m-1} \xrightarrow{d_{k}(\sigma)} \mathbf{n} \times \mathbf{m} 
\to \mathbf{m}
\end{equation*}
is surjective, so that $d_{k}(\sigma)$ is not in $\Delta^{n} \times
\partial\Delta^{m}$. Finally, $d_{k}(\sigma)$ is not a face of any
smaller non-degenerate $(m+n)$-simplices because it contains the
vertex $(n,0)$, so that $d_{k}(\sigma)$ is not a member of $S$.  The
map $i$ is therefore inner anodyne.
\end{proof}

%\nocite{J40}
%\nocite{Joyal-quasi-Kan}
%\nocite{GH}
%\nocite{FRGH}
%\nocite{GJ}

\nocite{pathcat}
\nocite{Lurie-HTT}
\nocite{Joyal-quasi-cat}

\bibliographystyle{plain} 
\bibliography{spt}

\end{document}